\documentclass[reqno,10pt]{amsart}
\usepackage{amsmath}	
\usepackage{amssymb}	
\usepackage{amsthm}
\usepackage{amsfonts}
\usepackage{latexsym}
\usepackage{mathrsfs}
\usepackage[all,ps,cmtip]{xy}
\usepackage{tabularx}
\usepackage{pict2e}
\usepackage{nicematrix}
\usepackage{graphicx}
\usepackage{xypic}
\usepackage{tikz-cd}
\usepackage{tikz}
\usetikzlibrary{matrix,arrows}
\usepackage[colorlinks=true, citecolor=red, urlcolor=blue, linkcolor=blue]{hyperref}
\theoremstyle{plain} 
\newtheorem{thm}{Theorem}[section]
\newtheorem{lem}[thm]{Lemma} 
\newtheorem{prop}[thm]{Proposition} 
\newtheorem{cor}[thm]{Corollary} 
\newtheorem{claim}[thm]{Claim}
\theoremstyle{definition} 
\newtheorem{defn}[thm]{Definition}
\newtheorem{rem}[thm]{Remark} 

\newtheorem{quest}[thm]{Question}

\makeatletter
\@namedef{subjclassname@2020}{%
  \textup{2020} Mathematics Subject Classification}
\makeatother

\newcommand{\A}{\mathcal{A}}

\newcommand{\T}{\mathcal{T}}
\newcommand{\Q}{\mathbb{Q}}
\newcommand{\Z}{\mathbb{Z}}
\newcommand{\ZN}{\mathbb{Z}}
\newcommand{\RN}{\mathbb{R}}
\newcommand{\PB}{\mathbb{P}}
\newcommand{\CO}{\mathcal{O}}
\newcommand{\CN}{\mathbb{C}}
\newcommand{\Ku}{\mathrm{Ku}}
\newcommand{\D}{\mathcal{D}^{\mathrm{b}}}
\newcommand{\Hom}{\mathrm{Hom}}

\newcommand{\Ext}{\mathrm{Ext}}
\newcommand{\ext}{\mathrm{ext}}
\newcommand{\ch}{\mathrm{ch}}
\newcommand{\td}{\mathrm{td}}
\newcommand{\Coh}{\mathrm{Coh}}
\newcommand{\Stab}{\mathrm{Stab}}
\newcommand{\PR}{\mathrm{pr}}

\newcommand{\RHom}{\mathrm{RHom}}
\newcommand{\RCH}{R\mathcal{H}om}
\newcommand{\CExt}{\mathcal{E}xt}
\newcommand{\CH}{\mathcal{H}}

\allowdisplaybreaks

\numberwithin{equation}{section}

\begin{document}

\title{Moduli of stable sheaves on quadric threefold}

\author{Song Yang}
\address{Center for Applied Mathematics and KL-AAGDM, Tianjin University, Weijin Road 92, Tianjin 300072, P.R. China}%
\email{syangmath@tju.edu.cn}%

\begin{abstract}
For each $0<\alpha<\frac{1}{2}$, there exists a  Bayer--Lahoz--Macr{\`{\i}}--Stellari  inducing Bridgeland stability condition $\sigma(\alpha)$ on a Kuznetsov component $\mathrm{Ku}(Q)$ of the smooth quadric threefold $Q$. 
We obtain the non-emptiness of the moduli space $M_{\sigma(\alpha)}([\mathcal{P}_{x}])$ of $\sigma(\alpha)$-semistable objects in $\mathrm{Ku}(Q)$ with the numerical class $[\mathcal{P}_{x}]$, where $\mathcal{P}_{x}\in \mathrm{Ku}(Q)$ is the projection sheaf of the skyscraper sheaf at a closed point $x\in Q$.
We show that the moduli space $\overline{M}_{Q}(\mathbf{v})$ of Gieseker semistable sheaves with Chern character $\mathbf{v}=\mathrm{ch}(\mathcal{P}_{x})$ is smooth and irreducible of dimension four, and prove that the moduli space $M_{\sigma(\alpha)}([\mathcal{P}_{x}])$ is isomorphic to $\overline{M}_{Q}(\mathbf{v})$. 
As an application, we show that the quadric threefold $Q$ can be reinterpreted as a Brill--Noether locus in the Bridgeland moduli space $M_{\sigma(\alpha)}([\mathcal{P}_{x}])$.
In the appendices, we show that the moduli space $M_{\sigma(\alpha)}([S])$ contains only one single point corresponding to the spinor bundle $S$ and give a Bridgeland moduli interpretation for the Hilbert scheme of lines in $Q$.
\end{abstract}

\date{\today}



\maketitle

\setcounter{tocdepth}{1}
\tableofcontents


\section{Introduction}

The notion of stability conditions on a $\CN$-linear triangulated category was introduced by Bridgeland in \cite{Bri07,Bri08}.
In the last decade, 
Bridgeland stability conditions and wall-crossing have became a very active area of research in algebraic geometry; 
see \cite{BM23} for a recent survey.
The purpose of this paper is to study moduli spaces of stable objects in a Kuznetsov component of the smooth quadric hypersurface in the complex projective four-space, 
by using Bridgeland stability conditions and wall-crossing. 


\subsection{Stability conditions on Kuznetsov components}

Recently,
Bayer--Lahoz--Macr{\`{\i}}--Stellari \cite{BLMS23} introduced a general construction of Bridgeland stability conditions on the right orthogonal complement, called a {\it Kuznetsov component}, of an exceptional collection in a $\CN$-linear triangulated category $\mathcal{D}$ with the Serre functor, by inducing from a weak stability condition on $\mathcal{D}$; see also \cite{BMMS12,LMS15,MS12} for a different construction in the case of smooth cubic threefolds and fourfolds.
Since then, using Bayer--Lahoz--Macr{\`{\i}}--Stellari's inducing construction, 
many Bridgeland stability conditions have been constructed on Kuznetsov components of Fano threefolds \cite{BLMS23,FP23,PR23,BF21}, cubic fourfolds \cite{BLMS23,LLMS18,LPZ23a}, Gushel--Mukai fourfolds \cite{PPZ22} and so on.
Recently, there is a rich emerging theory of moduli spaces of semistable objects on Kuznetsov components of Fano threefolds. 
It is now well-known that there is a classification \cite{IP99} of Fano threefolds of Picard rank $1$ with the index range from $1$ to $4$.
Via Bridgeland stability conditions and wall-crossing, 
the moduli spaces of semistable objects on Kuznetsov components of Fano threefolds of Picark rank $1$ of index $1$ and index $2$, and their categorical Torelli theorems have been widely investigated recently in \cite{APR22,PY22,Zha20,BBF+,PRo23,FP23,PR23,FLZ23,LZ22,LPZ23b,Qin23,JLLZ24,JLZ22,PS23,PPZ23} etc.
The moduli spaces of semistable objects on Kuznetsov components of the
remaining two cases, the rigid Fano threefolds of Picard rank $1$ of index $3$
and $4$ (namely, the smooth quadric hypersurface in $\PB^{4}$ and the complex
projective three-space $\PB^{3}$, respectively), have been studied less.

The current paper is focusing on the case of smooth quadric hypersurface in $\PB^{4}$.
Let $Q$ be the smooth quadric threefold.
Since $\{\CO,\CO(H)\}$ is an exceptional collection on the bounded derived category 
$\D(Q)=\D(\Coh(Q))$ of coherent sheaves on $Q$, there exists a semiorthogonal decomposition
$$
\D(Q)
=\langle \Ku(Q), \CO,\CO(H) \rangle,
$$
where $\CO$ is the structure sheaf of $Q$ and $\CO(H):=  \CO_{\PB^{4}}(H)|_{Q}$ and $H$ is a hyperplane. 
Here $\Ku(Q)$ is the right orthogonal complement of the exceptional collection $\{\CO, \CO(H)\}$, i.e., 
$$
\Ku(Q)= \{E\in \D(Q) \mid \Hom(\CO(iH),E[k])=0, \forall\, i=0,1 \textrm{ and } \forall\, k\in \Z\},
$$
which is called a {\it Kuznetsov component} (or {\it residual category}) of $Q$.
As a direct consequence of Bayer--Lahoz--Macr{\`{\i}}--Stellari's inducing construction, 
the Kuznetsov component $\Ku(Q)$ has a Bridgeland stability condition $\sigma(\alpha, \beta)$ as follows (see Proposition \ref{induced-stability-KuQ}): 
the values of $\alpha$ and $\beta$ vary in the subset of the upper $(\alpha,\beta)$-palne
$$
V:=   \{(\alpha, \beta)\in \RN_{>0} \times \RN \mid  -\frac{1}{2}\leq \beta <0, \alpha<-\beta; \; \textrm{or}\; -1< \beta < -\frac{1}{2}, \alpha\leq 1+\beta\}.
$$

\begin{prop}\label{key-BSC}
For each $(\alpha, \beta)\in V$, 
there exists a Bridgeland stability condition $\sigma(\alpha,\beta)$ on the Kuznetsov component $\Ku(Q)$.
\end{prop}

These stability conditions $\sigma(\alpha,\beta)$ parametrized by $V$ will be shown to be in the same $\tilde{\mathrm{GL}}^{+}_{2}(\RN)$-orbit in the stability manifold $\Stab(\Ku(Q))$ under the right $\tilde{\mathrm{GL}}^{+}_{2}(\RN)$-action (see Lemma \ref{GL-one-orbitV+}).
In fact, for a fixed $0<\alpha<\frac{1}{2}$, we set $\sigma(\alpha):=\sigma(\alpha, -\frac{1}{2})$,
we denote by the unique $\tilde{\mathrm{GL}}^{+}_{2}(\RN)$-orbit
$$
\mathcal{K}:=  \sigma(\alpha) \cdot \tilde{\mathrm{GL}}^{+}_{2}(\RN) \subset \Stab(\Ku(Q)).
$$ 
Then the stability condition $\sigma(\alpha,\beta)\in \mathcal{K}$ for all  $(\alpha,\beta)\in V$.


\subsection{Moduli spaces of stable objects in the Kuznetsov component}

By a theorem of Kapranov \cite{Kap88}, 
the Kuznetsov component $\Ku(Q)$ is generated by the exceptional collection $\{\CO(-H),S\}$,
where $S$ is the spinor bundle on $Q$.
We write the numerical classes of $\CO(-H)$ and $S$
$$
\lambda_{1}:=  [\CO(-H)]\; \textrm{ and }\;  \lambda_{2}:=  [S]
$$ 
in the numerical Grothendieck group $\mathcal{N}(\Ku(Q))$ of $\Ku(Q)$.
Then $\{\lambda_{1}, \lambda_{2} \}$ is a basis of $\mathcal{N}(\Ku(Q))$. 
Consider a numerical class $a\lambda_{1}+b\lambda_{2} \in \mathcal{N}(\Ku(Q))$ with $a, b\in \ZN-\{0\}$.
For a stability condition $\sigma\in \mathcal{K}$, we denote by 
$$
M_{\sigma}(a\lambda_{1}+b\lambda_{2})
$$ 
the moduli space of $\sigma$-semistable objects in $\Ku(Q)$ with the numerical class $a\lambda_{1}+b\lambda_{2}$.
Here by the moduli space $M_{\sigma}(a\lambda_{1}+b\lambda_{2})$ we mean the good moduli space in the sense of Alper \cite{Alp13} for the algebraic stack $\mathcal{M}_{\sigma}(a\lambda_{1}+b\lambda_{2})$ of $\sigma$-semistable objects in $\Ku(Q)$ with class $a\lambda_{1}+b\lambda_{2}$, 
whose existence follows from the work \cite{BLMNPS21}.
In fact, by \cite[Theorem 21.24]{BLMNPS21}, 
the moduli space $M_{\sigma}(a\lambda_{1}+b\lambda_{2})$ is a proper algebraic space over $\CN$.
Once a Bridgeland stability condition exists on the Kuznetsov component, 
one of the fundamental problems is to investigate the non-emptiness of the Bridgeland moduli spaces of semistable objects, 
and then the smoothness, irreducibility and projectivity and so on.
Naturally, one may ask the following: {\it is the moduli space $M_{\sigma}(a\lambda_{1}+b\lambda_{2})$ non-empty when its expected dimension positive}? {\it If non-empty, is it smooth and irreducible}?

This paper is mainly interested in moduli spaces of semistable objects with a primitive class $a\lambda_{1}+b\lambda_{2}$ with $\gcd(a,b)=1$. 
In particular, we consider the primitive numerical class $2\lambda_{2}-\lambda_{1}$ which is the class of the projection sheaf of the skyscraper sheaf supported on a closed point on $Q$.
Given a closed point $x\in Q$.
Let $\mathcal{P}_{x}\in \Ku(Q)$ be the sheaf given by the projection of a skyscraper sheaf $\CO_{x}$ supported on $x$. 
Here $\mathcal{P}_{x}$ is precisely defined by a short exact sequence 
$$
\xymatrix@C=0.5cm{
0 \ar[r]^{} & \mathcal{P}_{x} \ar[r]^{} & \CO^{\oplus 4} \ar[r]^{} & I_{x}(H) \ar[r]^{} & 0,}
$$
where $I_{x}$ is the ideal sheaf of the closed point $x$ in $Q$.
Moreover, there is a short exact sequence 
$$
\xymatrix@C=0.5cm{
0 \ar[r]^{} & \CO(-H) \ar[r]^{} & S^{\oplus 2} \ar[r]^{} & \mathcal{P}_{x} \ar[r]^{} & 0.}
$$
Hence, the numerical class $[\mathcal{P}_{x}]\in \mathcal{N}(\Ku(Q))$ of $\mathcal{P}_{x}$ is $2\lambda_{2}-\lambda_{1} $.
Now we can state the first result in this paper:

\begin{thm}\label{main-thm-nonempty}
For every stability condition $\sigma \in \mathcal{K}$, 
the moduli space $M_{\sigma}(2\lambda_{2}-\lambda_{1})$ is non-empty.
\end{thm}

To study the moduli space $M_{\sigma}(2\lambda_{2}-\lambda_{1})$  further,
we notice that the Chern character of $\mathcal{P}_{x}$ is 
$
\ch(\mathcal{P}_{x})=(3,-H,-\frac{1}{2}H^{2},\frac{1}{3}H^{3}).
$
The truncated Chern character $\ch_{\leq 2}(\mathcal{P}_{x})$ on $Q$ is the same as that on smooth cubic threefolds which has been detailed studied in \cite{BBF+}. 
Actually, we find that these two cases are extremely similar to each other.
This motivates us to investigate the moduli space, 
denoted by $\overline{M}_{Q}(\mathbf{v})$, 
of $H$-Gieseker semistable sheaves on $Q$ with Chern character 
$$
\mathbf{v}:=\ch(\mathcal{P}_{x})=(3,-H,-\frac{1}{2}H^{2},\frac{1}{3}H^{3}).
$$
As a matter of fact, the moduli space $\overline{M}_{Q}(\mathbf{v})$ is a projective moduli space parameterizing $S$-equivalence classes of $H$-Gieseker semistable sheaves with Chern character $\mathbf{v}$ (see \cite{Gie77,Mar77,Mar78,Sim94}).
Our main goal in this paper is to show the following result.

\begin{thm}\label{main-thm-iso}
The moduli space $\overline{M}_{Q}(\mathbf{v})$ is smooth and irreducible of dimension $4$.
Moreover, for every stability condition $\sigma \in \mathcal{K}$, 
the moduli space $M_{\sigma}(2\lambda_{2}-\lambda_{1})$ is isomorphic to $\overline{M}_{Q}(\mathbf{v})$.
\end{thm}

As an application, 
in Theorem \ref{Q-BN-locus}, 
we show that the smooth quadric threefold $Q$ can be recovered as a Brill--Noether locus in the Bridgeland moduli space $M_{\sigma}(2\lambda_{2}-\lambda_{1})$.

In Appendix \ref{Append}, 
we discuss the uniqueness of Bridgeland stable objects with the class of the spinor bundle $S$.
It is well-known that the spinor bundle $S$ is $\mu_{H}$-slope stable;
moreover, the moduli space $M_{Q}(-1,1)$ of  stable vector bundles of rank $2$ with Chern classes $c_{1}=-1$ and $c_{2}=1$ is consisting of one single point $S$ (cf. \cite{AS89,OS94}).
In Theorem \ref{spinor-bundle-Bstable}, we show that the moduli space $M_{\sigma}(\lambda_{2})$ also consists of one single point $S$.
In Appendix \ref{Append-Hilb-line}, we show that the moduli space $M_{\sigma}(\lambda_{2}-\lambda_{1})$ is isomorphic to the Hilbert scheme of lines in $Q$ (see Theorem \ref{Hilb-line-Q}).
It is well-known that the Hilbert scheme of lines is isomorphic to $\PB^{3}$.
Therefore, the moduli space $M_{\sigma}(\lambda_{2}-\lambda_{1})$ is isomorphic to $\PB^{3}$.


\subsection{Related work}
This work is inspired by \cite{BLMS23,BLMNPS21,BBF+} for constructing Bridgeland stability conditions and moduli spaces on a Kuznetsov component. 
Especially, benefiting from detailed comparing with smooth cubic threefolds in \cite{BBF+},
to study the moduli space $\overline{M}_{Q}(\mathbf{v})$, 
we use wall-crossing techniques for classifying sheaves with certain Chern characters on the smooth quadric threefold.

It is worth noticing that the non-emptiness problem of the Bridgeland moduli spaces in Kuznetsov components of quartic double solids and Gushel--Mukai threefolds have been solved recently by Perry--Pertusi--Zhao \cite{PPZ23}, and they showed that generically these moduli spaces are smooth and projective. 
In both cases, the Kuznetsov components are Enriques categories. 
In fact, they introduced a general approach to handle the case of Enriques categories.
However, for the smooth quadric threefold $Q$, the Kuznetsov component $\Ku(Q)$ is not an Enriques category.
It could be an interesting question to develop a geometric approach for addressing the non-emptiness problem of moduli spaces of $\sigma(\alpha,\beta)$-stable objects in $\Ku(Q)$ with a general numerical class $a\lambda_{1}+b\lambda_{2}$ where $\gcd(a,b)=1$. 
In particular, in \cite{Yan24}, we are interested in the moduli spaces $M_{\sigma}((n+1)\lambda_{2}-n\lambda_{1})$ and $\overline{M}_{Q}(\mathbf{v}_{n})$, where $\mathbf{v}_{n}:=(n+1)\ch(S)-n\ch(\CO(-H))$ for some $n\geq 2$.

In \cite{BF21}, Bolognese--Fiorenza constructed Bridgeland stability conditions on a different Kuznetsov component which is the right orthogonal complement of exceptional collection $\{\CO,\CO(H),\CO(2H)\}$ in $\D(Q)$ (see \cite[Remark 3.5]{BF21}), and applied them to reprove Kapranov's full exceptional collection theorem. 

In the case of $\PB^{3}$, via Bridgeland stability conditions and wall-crossing, 
Schmidt \cite{Sch20b,Sch23} recently showed that moduli spaces of semistable sheaves between rank zero and four with maximal third Chern character are smooth and irreducible.
It could be interesting to study the similar problem for the quadric threefold. 

In \cite{OS94}, 
Ottaviani--Szurek proved that the moduli space $M_{Q}(-1,2)$ of rank $2$ stable vector bundles with Chern classes $c_{1}=-1$ and $c_{2}=2$ is smooth and irreducible of dimension $6$. 
The moduli space $\overline{M}_{Q}(\mathbf{v})$ is closely related to the moduli space $M_{Q}(-1,2)$ (see Remark \ref{projection-construct-VB}).
Let $\overline{M}_{Q}(\nu)$ be the moduli space of $H$-Gieseker semistable sheaves with Chern character $\nu=(2,-H,-\frac{1}{2}H^{2},\frac{1}{3}H^{3})$.
Recently, Comaschi--Jardim \cite{CJ24} showed that the moduli space $\overline{M}_{Q}(\nu)$ consists of two irreducible components and one of the irreducible components contains $M_{Q}(-1,2)$.

Finally, 
the Fano threefolds of Picard rank $1$ of index $1$ and genus $\geq 6$, and the Fano threefolds of Picard rank $1$ of index $2$ and degree $\geq 2$ have been proved to be realized as a Brill--Noether locus of Bridgeland moduli spaces in their Kuznetsov components by Jacovskis--Liu--Zhang  \cite{JLZ22} and Feyzbakhsh--Liu--Zhang \cite{FLZ23}, respectively.
These works motivate the application in Section \ref{app-sect}.


\subsection*{Acknowledgements}
The author would like to thank Professor Guosong Zhao and Professor Xiaojun Chen for their constant encouragement, 
Xiangdong Yang and Xun Yu for many helpful conversations.
The author would like to thank Professor Paolo Stellari and Laura Pertusi for many helpful discussions;
especially when the author was visiting Universit\`{a} degli Studi di Milano (2018.1--2019.2) under the support of China Scholarship Council.
The author would like to thank the anonymous referee for careful reading the paper and  numerous valuable suggestions and comments.
This work is partially supported by the National Natural Science Foundation of China (No. 12171351).


\section{Bridgeland stability and tilt stability}

In this section, we recall the notions of slope stability and Gieseker stability for coherent sheaves, 
and mainly review the definitions and basic facts of (weak) Bridgeland stability conditions, tilt stability and wall-crossing.

\subsection{Slope stability}
Let $X$ be a smooth complex projective variety of dimension $3$ and $H$ an ample divisor on $X$.
Then the {\it $\mu_{H}$-slope} of $E\in \Coh(X)$ is given by
$$
\mu_{H}(E):=   \frac{H^{2}\cdot \ch_{1}(E)}{H^{3}\cdot \ch_{0}(E)}
$$
for $\ch_{0}^{\beta}(E)\neq 0$; otherwise, $\mu_{H}(E):=   +\infty$.

\begin{defn}
A sheaf $E\in \Coh(X)$ is called {\it $\mu_{H}$-slope (semi)stable}  if for any non-zero proper subsheaf $F\subset E$, the inequality $\mu_{H}(F)< (\leq)\, \mu_{H}(E/F)$ holds.
\end{defn}

A fundamental fact of  slope stability is that every $\mu_{H}$-slope semistable sheaf $E\in \Coh(X)$ satisfies the classical Bogomolov inequality  (cf. \cite[Theorem 7.3.1]{HL10}):
\begin{equation*} 
H \cdot \big(\ch_{1}(E)^{2}-2 \ch_{0}(E)\ch_{2}(E) \big)\geq 0.
\end{equation*}

\subsection{Gieseker stability}
Let $X$ be a smooth complex projective variety of dimension $3$ and $H$ an ample divisor on $X$.
For any sheaf $E \in \Coh(X)$, 
the {\it Hilbert polynomial} of $E$ with respect to $H$ is given by
$$
p(E,m):=   \chi(E(mH))=\sum_{i = 0}^{3} a_{i}(E) m^{i}.
$$
Moreover, we write $p_{2}(E,m):=  \sum_{i=1}^{3} a_{i}(E) m^{i}$.
We review a pre-order on the real coefficient polynomial ring $\RN[m]$: 
\begin{enumerate}
    \item[(1)] For every non-zero $f \in \RN[m]$, one has $f \prec 0$;
    \item[(2)] For non-zero $f, g \in \RN[m]$, if $\deg(f)>\deg(g)$, then $f \prec g$;
    \item[(3)] Suppose that two non-zero $f, g \in \RN[m]$ with $\deg(f) = \deg(g)$.
             Let $a_{f}$ and $a_{g}$ be the leading coefficient of $f$ and $g$. 
             Then $f \preceq g$ if and only if $\frac{f(m)}{a_{f}} \leq \frac{g(m)}{a_{g}}$ for all $m \gg 0$.
    \item[(4)] If $f, g \in \RN[m]$ with $f \preceq g$ and $g \preceq f$, one writes $f \asymp g$.
\end{enumerate}

\begin{defn}
\begin{enumerate}
    \item[(1)] A sheaf $E\in \Coh(X)$ is called {\it $H$-Gieseker (semi)stable} if for all non-zero proper subsheaves $F \subset E$, one has $p(F, m) \prec (\preceq)\, p(E, m)$.
    \item[(2)] A sheaf $E\in \Coh(X)$ is called {\it $2$-$H$-Gieseker (semi)stable} if for all non-zero proper subsheaves $F \subset E$, one has $p_{2}(F, m) \prec (\preceq) \, p_{2}(E/F,m)$.
\end{enumerate}
\end{defn}

Note that every Gieseker semistable sheaf is torsion-free.
For a torsion-free sheaf, by Riemann--Roch theorem, it follows that $\mu_{H}$-slope stability implies $H$-Gieseker stability and $H$-Gieseker semistability implies $\mu_{H}$-slope semistability.
Hence, these six notions yield each other as follows:
$$
\xymatrix@C=0.5cm{
\text{$\mu_{H}$-slope stable}  \ar@{=>}[r] & \text{$2$-$H$-Gieseker stable} \ar@{=>}[r] & \text{$H$-Gieseker stable} \ar@{=>}[d] \\
\text{$\mu_{H}$-slope semistable} & \text{$2$-$H$-Gieseker semistable} \ar@{=>}[l] & \text{$H$-Gieseker semistable} \ar@{=>}[l]
}
$$

\begin{rem}
If $\ch_{0}(E)$ and  $H^{2}\cdot \ch_{1}(E)/H^{3}$ are relatively prime,
then these six notions coincide.
\end{rem}

\subsection{(Weak) Bridgeland stability}

Let $\mathcal{D}$ be a $\CN$-linear triangulated category and $\A$ the heart of a bounded $t$-structure on $\mathcal{D}$.
We denote by  $K(\mathcal{D})$ (resp. $K(\A)$) the Grothendieck group of $\mathcal{D}$ (resp. $\A$). 
We know that the heart $\A$ of a bounded $t$-structure on $\mathcal{D}$ is an abelian category and the Grothendieck groups $K(\A)\cong K(\mathcal{D})$.
Roughly speaking, a (weak)  stability condition on $\mathcal{D}$ is consisting of the heart of a bounded $t$-structure and a (weak) stability function on it.

\begin{defn}
A {\it (weak) stability function} on  $\A$ is a pair $\sigma=(\A, Z)$ consists of a group homomorphism (called a central charge) $Z: K(\A) \rightarrow \CN$, $Z(E)=\Re Z(E) +i\, \Im Z(E)$
such that for every non-zero $E\in \A$, we have 
$\Im Z(E)\geq 0, \; \textrm{and}\; \Im Z(E)= 0 \Rightarrow \Re Z(E)\, (\leq ) < 0.$
\end{defn}

The (weak) stability function $\sigma=(\A, Z)$ allows one to define a notion of {\it $\mu_{\sigma}$-slope}
for any object $E\in \A$ by given
$$
\mu_{\sigma}(E)
:=   
\begin{cases}
-\frac{\Re Z(E)}{\Im Z(E)} & \textrm{if}\, \Im Z(E)>0; \\
+\infty &  \textrm{otherwise,}
\end{cases}
$$
and a notion of {\it $\sigma$-stability} on $\mathcal{D}$.

\begin{defn}
A non-zero object $E\in \mathcal{D}$ is called {\it $\sigma$-(semi)stable} if some shift $E[k]\in \A$ and for every non-zero proper subobject $F\subset E[k]$ in $\A$, the inequality $\mu_{\sigma}(F) \, (\leq) < \mu_{\sigma}(E[k]/F)$ holds.
\end{defn}

Fix a finite rank lattice $\Lambda$ and a surjective homomorphism $\upsilon \colon K(\mathcal{D})\rightarrow \Lambda$.

\begin{defn}
A {\it (weak) stability condition} on $\mathcal{D}$ (with respect to $\Lambda$) is a pair $\sigma=(\A, Z)$ consisting of the heart  $\A$ of a bounded $t$-structure on $\mathcal{D}$ and a group homomorphism $Z: \Lambda \rightarrow \mathbb{C}$ such that the following conditions hold:
\begin{enumerate}
\item[(i)] The composition $\xymatrix@C=0.3cm{K(\A) \ar[r]^{\;\;\;\upsilon} & \Lambda \ar[r]^{} &\mathbb{C}}$
defines a (weak) stability function on $\A$; for any object $E\in \mathcal{D}$, we write $Z(E):=   Z(\upsilon(E))$ for simplicity.
\item[(ii)] (HN-filtrations) Any object of $\A$ has a Harder--Narasimhan filtration in $\sigma$-semistable ones.

\item[(iii)] (Support property) There is a quadratic form $\Delta$ on $\Lambda\otimes \mathbb{R}$ such that $\Delta|_{\ker Z}$ is negative definite, and $\Delta(E)\geq 0$ for all $\sigma$-semistable objects $E\in \A$.
\end{enumerate}
\end{defn}

We denote by $\Stab_{\Lambda}(\mathcal{D})$ the set of stability conditions on the $\CN$-linear triangulated category $\mathcal{D}$.
This set $\Stab_{\Lambda}(\mathcal{D})$, called the stability manifold of $\mathcal{D}$, has the structure of a complex manifold \cite[Theorem 1.2]{Bri07}; see also \cite[Proposition 11.5]{BMS16} and \cite[Theorem 1.2]{Bay19}.
Moreover, there exist two natural group actions on the stability manifold $\Stab_{\Lambda}(\mathcal{D})$:
\begin{itemize}
  \item a right action of the universal covering space $\tilde{\mathrm{GL}}^{+}_{2}(\RN)$ of $\mathrm{GL}^{+}_{2}(\RN)$.
  \item a left action the group of $\CN$-linear exact autoequivalences of $\mathcal{D}$.
\end{itemize}
We refer to \cite[Lemma 8.2]{Bri08} for detailed discussions.


\subsection{Tilt stability}

The approach by tilting a heart provides an important tool for constructing (weak) stability conditions.
Given a weak stability condition $\sigma=(\A, Z)$ on a $\CN$-linear triangulated category $\mathcal{D}$.
For every a real number $\mu\in \mathbb{R}$, 
there is a torsion pair $(\T_{\sigma}^{\mu}, \mathcal{F}_{\sigma}^{\mu})$ defined as extension closures:
$$
\T_{\sigma}^{\mu}
:=   \langle 
E \in \A \mid E \,\textrm{is}\, \sigma\textrm{-semistable with}\, \mu_{\sigma}(E)>\mu 
\rangle,
$$
$$
\mathcal{F}_{\sigma}^{\mu}
:=   \langle 
E \in \A \mid E \,\textrm{is}\, \sigma\textrm{-semistable with}\, \mu_{\sigma}(E)\leq\mu
\rangle.
$$

\begin{prop}[\cite{HRS96}]
The extension closure of $\T_{\sigma}^{\mu}$ and $\mathcal{F}_{\sigma}^{\mu}$,
$$
\A^{\mu}_{\sigma}:=  \langle  \T_{\sigma}^{\mu}, \mathcal{F}_{\sigma}^{\mu}[1]\rangle,
$$
forming the heart of a bounded $t$-structure on $\mathcal{D}$.
\end{prop}

This abelian category $\A_{\sigma}^{\mu}$ is called the heart obtained by {\it tilting} $\A$ with respect to $\sigma$ at the slope $\mu$.

In what follows, we will recall the first tilting construction for defining tilt stability.
From now on, let $X$ be a smooth complex projective variety of dimension $3$ and $H$ an ample divisor on $X$.
For every integer $j\in \{0,1, 2, 3\}$, 
we consider the lattice $\Lambda_{H}^{j}\cong \Z^{j+1}$ generated by 
$$
(H^{3}\ch_{0}, H^{2}\ch_{1}, \cdots, H^{3-j}\ch_{j})\in \Q^{j+1}
$$
with the map $\upsilon_{H}^{j}: K(\D(X))\rightarrow \Lambda_{H}^{j}$.

\begin{lem}[{\cite[Example 2.8]{BLMS23}}]
The pair $\sigma_{H}:=   (\Coh(X), Z_{H})$ with the central charge
$$
Z_{H}(E):=  -H^{2} \cdot \ch_{1}(E)+ i H^{3} \cdot \ch_{0}(E)
$$
is a weak stability condition on $\D(X)$ (with respect to $\Lambda_{H}^{1}$) with the trivial quadratic form.
\end{lem}

Now given a real number $\beta\in \RN$,
we denote by $\Coh^{\beta}(X)$ the heart of a bounded $t$-structure on $\D(X)$ obtained by tilting weak stability condition $\sigma_{H}=(\Coh(X), Z_{H})$ at the slope $\mu_{H}=\beta$, i.e., 
$\Coh^{\beta}(X)=\langle \T_{H}^{\beta}, \mathcal{F}_{H}^{\beta}[1] \rangle$, where
$$
\T_{H}^{\beta}
:=   \langle E \in \Coh(X) \mid E \,\textrm{is}\; \mu_{H}\textrm{-semistable with}\, \mu_{H}(E)>\beta \rangle,
$$
$$
\mathcal{F}_{H}^{\beta}
:=   \langle E \in  \Coh(X) \mid E \,\textrm{is}\; \mu_{H}\textrm{-semistable with}\, \mu_{H}(E)\leq\beta \rangle.
$$

For an object $E\in \D(X)$, the twisted Chern character $\ch^{\beta}(E)$ is defined to be $\ch^{\beta}(E):=e^{-\beta H}\cdot \ch(E)$.
More explicitly, one has the first three formulae:
\begin{eqnarray*}
\ch_{0}^{\beta}(E) & = & \ch_{0}(E) \\
\ch_{1}^{\beta}(E) & = & \ch_{1}(E)-\beta H\cdot \ch_{0}(E) \\
\ch_{2}^{\beta}(E) & = & \ch_{2}(E)-\beta H\cdot \ch_{1}(E)+\frac{\beta^{2}}{2}H^{2} \cdot \ch_{0}(E) \\
\end{eqnarray*}
For any $(\alpha, \beta)\in \RN_{>0}\times \RN$,
one defines a central charge
\begin{equation}\label{central-charge-wsc}
Z_{\alpha, \beta}(E)
:=   \frac{1}{2}\alpha^2 H^{3}\cdot \ch_{0}^{\beta}(E)-H \cdot \ch_{2}^{\beta}(E)
+ i H^{2}\cdot \ch_{1}^{\beta}(E)
\end{equation}
for  any $E\in \Coh^{\beta}(X)$.

\begin{prop}[{\cite[Proposition 2.12]{BLMS23}}]
The pair $\sigma_{\alpha, \beta}:= (\Coh^{\beta}(X), Z_{\alpha, \beta})$ is a weak stability condition on $\D(X)$ (with respect to $\Lambda_{H}^{2}$) with the quadratic form given by the $H$-discriminant $\Delta_{H}$.
Moreover, these stability conditions vary continuously as $(\alpha,\beta)$ varies in $\RN_{>0}\times \RN$.
\end{prop}

Here the {\it $H$-discriminant} of $E\in \D(X)$ is given by
\begin{eqnarray*}
\Delta_{H}(E): 
& = & (H^{2}\cdot \ch^{\beta}_{1}(E))^{2}-2(H^{3}\cdot\ch^{\beta}_{0}(E))\cdot (H \cdot \ch^{\beta}_{2}(E)) \\ 
& = & (H^{2}\cdot \ch_{1}(E))^{2}-2(H^{3}\cdot\ch_{0}(E))\cdot (H \cdot \ch_{2}(E)).
\end{eqnarray*}
One of the fundamental results is that  there is a tilt stability version  of the Bogomolov inequality (see \cite[Corollary 7.3.2]{BMT14} and \cite[Theorem 3.5]{BMS16}):

\begin{thm}[\cite{BMT14}]\label{B-ineq}
If $E\in \Coh^{\beta}(X)$ is $\sigma_{\alpha,\beta}$-semistable,
then its $H$-discriminant $\Delta_{H}(E)\geq 0$.
\end{thm}

In \cite{BMT14},  the weak stability condition $\sigma_{\alpha,\beta}$ is called {\it tilt stability}.
In case of surfaces, it indeed defines a Bridgeland stability condition (see \cite{Bri08,AB13}).
The tilt stability is strongly connected to $\mu_{H}$-slope stability and the notion of $2$-$H$-Gieseker stability.
More precisely, the notion of $2$-$H$-Gieseker stability occurs as a large volume limit of tilt stability (see \cite[Proposition 4.8]{BBF+} and \cite[Lemma 2.7]{BMS16}):

\begin{prop}[{\cite{Bri08}}]\label{limit-tilt-stab-Gie-stab-prop}
Suppose $E\in \D(X)$ and $\mu_{H}(E)>\beta$.
Then $E\in \Coh^{\beta}(X)$ and $E$ is $\sigma_{\alpha,\beta}$-(semi)stable for $\alpha\gg 0$
if and only if $E$ is a $2$-$H$-Gieseker (semi)stable sheaf.
\end{prop}

In particular, one has the following:

\begin{cor}\label{limit-tilt-stab-Gie-stab}
If $\gcd(\ch_{0}(E),\frac{H^{2}\cdot \ch_{1}(E)}{H^{3}})=1$ and $\mu_{H}(E)>\beta$, 
then $E\in \Coh^{\beta}(X)$ and $E$ is $\sigma_{\alpha,\beta}$-stable for $\alpha\gg 0$
if and only if $E$ is a $\mu_{H}$-slope stable sheaf.
\end{cor}

\subsection{Wall-and-chamber structures in tilt stability}

\begin{defn}\label{numerical-wall-def}
Let $v\in K(\D(X))$.
\begin{enumerate}
  \item[(i)] A {\it numerical wall} for $v$ with respect to $w\in K(\D(X))$ in tilt-stability is a non-trivial proper subset in the upper half plane
$$
W(v,w):=  \{ (\alpha, \beta)\in \RN_{>0}\times \RN \mid \mu_{\alpha,\beta}(v)=\mu_{\alpha,\beta}(w)\}.
$$
  \item[(ii)] A {\it chamber} for the class $v$ is a connected component in the complement of the union of walls in the upper half $(\alpha,\beta)$-plane. 
\end{enumerate}
\end{defn}

The numerical walls in tilt stability have well-behaved wall-and-chamber structure.
The properties of numerical walls satisfy the following Nested wall theorem which was first proved in \cite{Mac14a}; see also for example \cite[Theorem 3.1]{Sch20a}. 

\begin{thm}[]\label{tilt-wall-struct-thm}
Let $v \in K(\D(X))$ with $\Delta_{H}(v) \geq 0$.
All numerical walls for $v$ in the $(\alpha,\beta)$-plane is described as follows:
\begin{enumerate}
\item[(1)] A numerical wall for $v$ is either a semicircle centered at $\beta$-axis or a vertical line parallel to $\alpha$-axis in the upper half plane.
\item[(2)] If $\ch_{0}(v) \neq 0$, then there exists a unique numerical vertical wall for $v$ given by $\beta=\mu_{H}(v)$. 
The remaining numerical walls for $v$ are consisted of two sets of nested semicircular walls whose apexes lie on the hyperbola $\mu_{\alpha, \beta}(v) = 0$. 
\item[(3)] If $\ch_{0}(v) = 0$ and $H^{2} \cdot \ch_{1}(v) \neq 0$, 
then every numerical wall for $v$ is a semicircle whose top point lies on the vertical line $\beta=\frac{H \cdot \ch_{2}(v)}{H^{2} \cdot \ch_{1}(v)}$ in the upper half plane.
\item[(4)] Any two distinct numerical walls intersect empty.
\end{enumerate}
\end{thm}

Suppose that $\xymatrix@C=0.4cm{0 \ar[r]^{} & A \ar[r]^{} & E \ar[r]^{} & B \ar[r]^{} & 0}$ is a short exact sequence of $\sigma_{\alpha,\beta}$-semistable objects in $\Coh^{\beta_{0}}(X)$ for some $(\alpha_{0},\beta_{0})\in W(E,A)$.  
By Theorem \ref{tilt-wall-struct-thm}, this exact sequence is a short exact sequence of $\sigma_{\alpha,\beta}$-semistable objects in $\Coh^{\beta}(X)$ for all $(\alpha,\beta)\in W(E,A)$.

\begin{defn}\label{wall-def}
A numerical wall $W$ for $v\in K(\D(X))$ is called an {\it actual wall} for $v$ if there exists a short exact sequence of $\sigma_{\alpha,\beta}$-semistable objects
\begin{equation}\label{def-wall-sequ}
\xymatrix@C=0.5cm{
0 \ar[r]^{} & A \ar[r]^{} & E \ar[r]^{} & B \ar[r]^{} & 0}
\end{equation}
in $\Coh^{\beta}(X)$ for one $(\alpha,\beta)\in W(E,A)$ such that $v=\ch(E)$ and $\mu_{\alpha,\beta}(A)=\mu_{\alpha,\beta}(E)$ defines the numerical wall $W$ (i.e., $W=W(E,A)$). 
\end{defn}

\begin{rem}\label{actual-restrict}
In Definition \ref{wall-def}, 
an actual wall satisfies the following constraints:
\begin{enumerate}
  \item[(1)] $\mu_{\alpha,\beta}(A)=\mu_{\alpha,\beta}(E)=\mu_{\alpha,\beta}(B)$;
  \item[(2)] $\Delta_{H}(A)\geq 0, \Delta_{H}(B)\geq 0$;
  \item[(3)] $\Delta_{H}(A)\leq \Delta_{H}(E)$, $\Delta_{H}(B)\leq \Delta_{H}(E)$.
\end{enumerate} 
Moreover, we have the following
$$
\Delta_{H}(A)+\Delta_{H}(B)\leq \Delta_{H}(E),
$$
and the equality holds if either $A$ or $B$ is a sheaf supported in points (see \cite[Proposition 12.5]{BMS16}).
\end{rem}


\section{Quadric threefold}

In this section, we review some basic facts on the smooth quadric threefold.
Let $Q$ be the smooth quadric threefold  in $\PB^{4}$ and $H$ a general hyperplane section on $Q$.
One may view $Q$ as a smooth hyperplane section $\mathrm{Gr}(2,4)\cap \PB^{4}$ of $\mathrm{Gr}(2,4)$ in $\PB^{5}$.
The spinor bundle $S$ on $Q$ is the restriction of the universal bundle on $\mathrm{Gr}(2,4)$. 
Hence, the spinor bundle $S$ is a rank $2$ vector bundle on $Q$.
We first list some basic facts on the spinor bundle $S$; we refer to \cite{Ott88} for more details:
\begin{itemize}
  \item there is a short exact sequence 
          $
          \xymatrix@C=0.3cm{
          0 \ar[r]^{} & S \ar[r]^{} & \CO^{\oplus 4}  \ar[r]^{} & S(H) \ar[r]^{} & 0;}
          $
  \item the dual bundle $S^{\vee}\cong S(H)$ and $S$ is $\mu_{H}$-slope stable;
  \item the Chern classes $c_{1}=-1$, $c_{2}=1$ and $c_{3}=0$; 
           hence, the Chern character of $S$ is $\ch(S)=(2, -H, 0, \frac{1}{12}H^{3})$.
\end{itemize}

On the one hand, since the canonical sheaf $\omega_{Q}=\CO(-3H)$, hence the Serre functor on $\D(Q)$ is 
$$
S_{Q}(-)=(-) \otimes \omega_{Q}[3]=(-) \otimes \CO(-3H)[3].
$$
By Kodaria vanishing and Serre duality,
the collection $\{\CO, \CO(H)\}$ is exceptional and hence there is a semiorthogonal decomposition
$$
\D(Q)
=\langle \Ku(Q), \CO,\CO(H) \rangle.
$$
By a theorem of Kapranov \cite{Kap88}, 
the sequence $\{\CO(-H),S,\CO,\CO(H)\}$ is a strongly full exceptional collection and thus $\Ku(Q)= \langle \CO(-H), S\rangle$.
Then, there is a projection functor 
$$
\PR: \D(Q) \rightarrow \Ku(Q)
$$ 
which is given by $\PR(-):=  \mathrm{L}_{\CO} \circ \mathrm{L}_{\CO(H)}(-)[2]$,
where the left mutation $\mathrm{L}_{E}$ for an exceptional object $E\in \D(Q)$ is given by the exact triangle
$$
\xymatrix@C=0.4cm{
\RHom(E,F)\otimes E \ar[r]^{\;\;\;\;\;\;\;\;\;\;\;\;\;\;\;\mathrm{ev}} &  F \ar[r]^{} & \mathrm{L}_{E}(F)}.
$$
Here $\mathrm{ev}$ is the evaluation morphism.

On the other hand, we note that the Todd class of $Q$ is given by 
$$
\td(Q)=1+\frac{3}{2}H+\frac{13}{12}H^{2}+\frac{1}{2}H^{3}.
$$
By Hirzebruch--Riemann--Roch theorem,
the holomorphic Euler--Poincar\'{e} characteristic of $E\in \D(Q)$ is given by the formula
\begin{equation}\label{HRR-quadric3}
\chi(E)=\ch_{3}(E)+\frac{3}{2}H\cdot \ch_{2}(E)+\frac{13}{12}H^{2}\cdot \ch_{1}(E)+\frac{1}{2}H^{3} \ch_{0}(E).
\end{equation}

The numerical Chow group of the quadric threefold $Q$ is given as follows:

\begin{lem}\label{Num-Chow-Q}
The numerical Chow group has a basis $\{1,H,\frac{1}{2}H^{2}, \frac{1}{2}H^{3}\}$.
Moreover, if $E\in \D(Q)$, then $\ch_{2}(E)\in \frac{H^{2}}{2}\cdot \ZN$ and $\ch_{3}(E)\in \frac{H^{3}}{12}\cdot \ZN$.
\end{lem}

\begin{proof}
Note that $\mathrm{Pic}(Q)=H^{2}(Q,\ZN) \cong \ZN\cdot H$.
The cohomology ring $H^{\ast}(Q,\ZN)$ is generated by $H$, a line $[\ell]\in H^{4}(Q,\ZN)$ and a point $[p]\in H^{6}(Q,\ZN)$ with the relation $H^{2}=2\ell$, $H\cdot l=p$ and $H^{3}=2p$.
 Since $\ch_{2}(E)=\frac{1}{2}(c_{1}^{2}(E)-2c_{2}(E))$, 
 so $\ch_{2}(E)\in \frac{H^{2}}{2}\cdot \ZN$.
 Finally, $\ch_{3}(E)\in \frac{H^{3}}{12}\cdot \ZN$ follows from the Hirzebruch--Riemann--Roch formula \eqref{HRR-quadric3}.
\end{proof}

Likewise to the case of smooth cubic threefolds (see \cite[Proposition 4.20]{BBF+}), 
direct sums of line bundles on the smooth quadric threefold can also be detected among semistable sheaves (or objects) in terms of their Chern characters. 

\begin{prop}\label{sum-lines-detect}
Let $Q$ be the smooth quadric threefold. 
\begin{enumerate}
\item[(i)]  If $E$ is $\mu_{H}$-slope semistable (or $\sigma_{\alpha, \beta}$-semistable for some $\alpha > 0$ and
                $\beta < 0$) with $\ch(E) = (r, 0, 0, eH^{3})$ where $r > 0$, then $e \leq 0$. 
               Moreover, if $e = 0$, then $E \cong \CO^{\oplus r}$.
\item[(ii)]  If $E$ is $\sigma_{\alpha, \beta}$-semistable for some $\alpha > 0$ and $\beta > 0$ 
                with $\ch(E) = (-r, 0, 0, eH^{3})$ where $r > 0$, then $e = 0$ and $E \cong \CO^{\oplus r}[1]$.
\end{enumerate}
\end{prop}

\begin{proof} 
We follow the proof of \cite[Proposition 4.20]{BBF+}.
In both cases,
by Theorem \ref{tilt-wall-struct-thm}, 
$E$ has no semicircular walls.
We first show that if $E$ is a $\mu_{H}$-slope stable reflexive sheaf with Chern character $\ch(E)=(r,0,0,eH^3)$, then $E\cong \CO$. 
If $\beta\in (-3,0)$, 
then $\mu_{H}(E)=0>\beta>-3=\mu_{H}(\CO(-3H))$.
Since $E$ and $\CO(-3H)$ are $\mu_{H}$-slope stable torsion sheaves,
by Proposition \ref{limit-tilt-stab-Gie-stab-prop}, $E, \CO(-3H)[1]\in \Coh^{\beta}(Q)$ for $\beta\in (-3,0)$.
Similarly, if $\beta\in (0,3)$, then $E[1], \CO(3H) \in \Coh^{\beta}(Q)$ for $\beta\in (0,3)$.
By \cite[Proposition 4.18]{BBF+}, 
$E[1]$ is $\sigma_{\alpha,0}$-stable.
Since $E$ has no semicircular walls, hence $E$ is $\sigma_{\alpha,\beta}$-stable for $\alpha>0$ and $\beta<0$  and $E[1]$ is $\sigma_{\alpha,\beta}$-stable for $\alpha>0$ and $\beta\geq 0$.
Note that 
$
\displaystyle 
\lim_{\alpha\to 0^{+}}\mu_{\alpha,\beta}(E)=-\frac{\beta}{2} > -\frac{\beta}{2}-\frac{3}{2}=\lim_{\alpha\to 0^{+}}\mu_{\alpha, \beta}(\CO(-3H)[1]).
$ 
Since $E$ and $\CO(-3H)[1]$ are $\sigma_{\alpha, \beta}$-stable for $\alpha \ll 1$ and $\beta\in (-3,0)$, by Serre duality, one has
$$
\Ext^{2}(\CO, E) \cong \Hom(E, \CO(-3H)[1])=0.
$$ 
Likewise, 
$
\displaystyle \lim_{\alpha\to 0^{+}}\mu_{\alpha, \beta}(\CO(3H))
=\frac{3}{2}-\frac{\beta}{2}>-\frac{\beta}{2}
=\lim_{\alpha\to 0^{+}}\mu_{\alpha,\beta}(E[1]).
$
Since $E[1]$ and $\CO(3H)$ are $\mu_{\alpha, \beta}$-stable for $\alpha \ll 1$ and $\beta \in (0, 3)$, one has 
$$
\Ext^{2}(E, \CO) \cong \Hom(\CO(3H), E[1])=0.
$$
By Hirzebruch--Riemann--Roch formula \eqref{HRR-quadric3}, 
we have $\chi(\CO,E)=r+2e$ and $\chi(E,\CO)=r-2e$.
Note that one of $r+2e$ and $r-2e$ must be positive; 
otherwise $2e\leq -r$ and $r\leq 2e$ contradict $r>0$.
Hence, either $\Hom(E,\CO)$ or $\Hom(\CO,E)$ is non-trivial.
Since $E$ and $\CO$ are reflexive and slope-stable of slope $0$, 
so one has $E \cong \CO$.
The rest arguments are the same as that of \cite[Proposition 4.20]{BBF+}.
\end{proof}


\section{Construction of stability conditions}

In order to construct Bridgeland stability conditions on a Kuznetsov component,
following the work \cite{BLMS23}, one needs the second tilting construction obtained from the tilt stability.

Fix a real number $\mu\in \RN$ and a unit vector $u$ in the upper half plane such that $\mu=-\frac{\Re u}{\Im u}$,  
we denote by the heart of a bounded $t$-structure 
$$
\Coh_{\alpha,\beta}^{\mu}(Q) =\langle \T_{\alpha,\beta}^{\mu}, \mathcal{F}_{\alpha,\beta}^{\mu}[1]\rangle,
$$
which is obtained by tilting the tilt stability $\sigma_{\alpha,\beta}=(\Coh^{\beta}(Q),Z_{\alpha,\beta})$
at the slope $\mu_{\alpha,\beta}=\mu$.
Thanks to \cite[Proposition 2.15]{BLMS23},  the pair 
$$
\sigma_{\alpha,\beta}^{\mu}:=  (\Coh_{\alpha,\beta}^{\mu}(Q),Z_{\alpha,\beta}^{\mu})
$$
is a weak stability condition on $\D(Q)$, where $Z_{\alpha,\beta}^{\mu}:=  \frac{1}{u}Z_{\alpha,\beta}$ and $Z_{\alpha,\beta}$ is given by \eqref{central-charge-wsc}.

Next, we consider the following subset of the upper half $(\alpha,\beta)$-plane:
$$
\widetilde{V}:=   \{(\alpha, \beta)\in \RN_{>0} \times \RN \mid  -1\leq \beta <0, \alpha<-\beta; \; \textrm{or}\; -2< \beta < -1, \alpha\leq 2+\beta\}.
$$
For constructing stability conditions on $\Ku(Q)$, we fix some notation:
\begin{itemize}
\item $\A(\alpha,\beta):= \Coh_{\alpha,\beta}^{0}(Q)\cap \Ku(Q)$;
\item $Z(\alpha,\beta):=Z_{\alpha,\beta}^{0}|_{\Ku(Q)}$, where $Z_{\alpha,\beta}^{0}=-i Z_{\alpha,\beta}$;
\item  the lattice 
$
\Lambda_{H,\Ku(Q)}^{2}:=   \mathrm{Im}(K(\Ku(Q))\rightarrow K(Q) \rightarrow \Lambda_{H}^{2}) \cong \ZN^{\oplus 2}.
$
\end{itemize}

\begin{prop}[\cite{BLMS23}]\label{induced-stability-KuQ}
For each $(\alpha, \beta)\in \widetilde{V}$, 
the pair 
$$
\sigma(\alpha, \beta):=   (\A(\alpha,\beta),Z(\alpha,\beta))
$$
is a Bridgeland stability condition on $\Ku(Q)$ with respect to the lattice $\Lambda_{H,\Ku(Q)}^{2}$.
\end{prop}

\begin{proof}
By applying Serre functor to $\CO$ and $\CO(H)$, 
we have  $S_{Q}(\CO)=\CO(-3H)[3] \in\Coh(Q)[3]$
and $S_{Q}(\CO(H))=\CO(-2H)[3]\in\Coh(Q)[3]$.
Note that the $\mu_{H}$-slopes
$$\mu_{H}(\CO) =0 > \beta,
\mu_{H}(\CO(H))=1, \mu_{H}(\CO(-3H))=-3
\textrm{ and } \mu_{H}(\CO(-2H))=-2.
$$
Since line bundles are $\mu_H$-slope stable, 
\cite[Lemma 2.7]{BMS16} yields that $\CO$, $\CO(H)$, $\CO(-3H)[1]$ and $\CO(-2H)[1]$ are in $\Coh^{\beta}(Q)$ for $-2<\beta<0$. 
Since the $H$-discriminant of a line bundle is zero, 
by \cite[Corollary 3.11 (a)]{BMS16}, 
$\CO$,  $\CO(H)$, $\CO(-3H)[1]$ and $\CO(-2H)[1]$ 
are $\sigma_{\alpha, \beta}$-stable for $\alpha>0$.

For any $(\alpha, \beta)\in \widetilde{V}$, by direct computations, we have
\begin{eqnarray*}
\Re Z_{\alpha, \beta}(\CO) 
&=& \frac{1}{2}\alpha^{2}H^{3}-\frac{\beta^{2}}{2}H^{3}<0,\\
\Re Z_{\alpha, \beta}(\CO(H)) 
&=& \frac{1}{2}\alpha^{2}H^{3}-\frac{(1-\beta)^{2}}{2}H^{3}<0, \\
\Re Z_{\alpha, \beta}(\CO(-3H)[1]) 
&=& -\frac{1}{2}\alpha^{2}H^{3}+\frac{(3+\beta)^{2}}{2}H^{3}>0, \\
\Re Z_{\alpha, \beta}(\CO(-2H)[1]) 
&=&  -\frac{1}{2}\alpha^{2}H^{3}+\frac{(2+\beta)^{2}}{2}H^{3} \geq 0.
\end{eqnarray*}
Since $\Im Z_{\alpha, \beta}>0$ for all objects $\CO$, $\CO(H)$, $\CO(-3H)[1]$ and $\CO(-2H)[1]$, 
so we have
$$
\mu_{\alpha, \beta}(\CO),  
\mu_{\alpha, \beta}(\CO(H))>0,
\textrm{ and }
\mu_{\alpha, \beta}(\CO(-3H)[1])<0,  
\mu_{\alpha,\beta}(\CO(-2H)[1]) \leq 0.
$$
As a result, for each $(\alpha, \beta)\in \widetilde{V}$,
$\CO$, $\CO(H)$, 
$\CO(-3H)[2]$ and $\CO(-2H)[2]$ are contained in the heart $\Coh_{\alpha, \beta}^{0}(Q)$. 

We denote by $\Coh^{\beta}(Q)_{0}$ the category of $0$-dimensional torsion sheaves which are the only objects in $\Coh^{\beta}(Q)$ with the central charge $Z_{\alpha,\beta}=0$. 
Since the global sections of $0$-dimensional torsion sheaves are non-trivial, 
hence $\Coh^{\beta}(Q)_{0}\cap \Ku(X)=\emptyset$.
According to \cite[Lemma 2.16]{BLMS23},
there is no non-zero object $F\in \Coh_{\alpha, \beta}^{0}(Q) \cap \Ku(Q)$ with $Z_{\alpha, \beta}^{0}(F)=0$. 
As a result, \cite[Proposition 5.1]{BLMS23} concludes the proof.
\end{proof}

\begin{rem}
As triangulated categories, the Kuznetsov component $\Ku(Q)$ is equivalent to the bounded derived category $\D(K(4))$ of the Kronecker quiver $K(4)$. 
Here the Kronecker quiver $K(l)$ is the quiver with two vertices and $l$ parallel arrows.
By a theorem of Macr{\`{\i}} \cite{Mac07}, the stability manifold $\Stab(\D(K(l)))$ is simply-connected.  
In a recent paper \cite{DK19},
Dimitrov--Katzarkov proved that for the Kronecker quiver $K(l)$ with $l\geq 3$,
the stability manifold $\Stab(\D(K(l)))$ is biholomorphic to $\CN\times \mathcal{H}$, 
where $\mathcal{H}= \{z\in \CN \mid \Im(z)>0\}$.
Therefore, the stability manifold $\Stab(\Ku(Q))$ is also biholomorphic to $\CN\times \mathcal{H}$.
\end{rem}

On the other hand, for studying the moduli spaces of stable objects in $\Ku(Q)$,
we would like to consider the following subset $V$ of $\widetilde{V}$,
$$
V:=   \{(\alpha, \beta)\in \widetilde{V} \mid  -\frac{1}{2} \leq \beta <0, \alpha<-\beta; \; \textrm{or}\; -1< \beta < -\frac{1}{2}, \alpha\leq 1+\beta\}.
$$
Note that the numerical Grothendieck group $\mathcal{N}(\Ku(Q))$ has a basis as follows:
$$
 \lambda_{1}=[\CO(-H)]=1-H+\frac{1}{2}H^{2}-\frac{1}{6}H^{3},\;
 \lambda_{2}=[S]=2-H+\frac{1}{12}H^{3}
$$
For every $(\alpha,\beta)\in V$, the image of the stability function $Z(\alpha,\beta)$ is not contained in a line.
Moreover, the determinant
$$
\left|\begin{array}{cc}
-(\beta+1) & -(2\beta+1) \\
\frac{1}{2}+\beta+\frac{\beta^{2}}{2}-\frac{\alpha^{2}}{2} & \beta^{2}+\beta-\alpha^{2}
\end{array}\right|=\frac{1}{2}((\beta+1)^{2}+\alpha^{2})>0.
$$  
Proving pretty much exactly the same as that for \cite[Proposition 3.6]{PY22},
we have the following:

\begin{lem}\label{GL-one-orbitV+}
Fix $0<\alpha_{0}<\frac{1}{2}$.
For any $(\alpha, \beta)\in V$,
there exists $\tilde{g}\in \tilde{\mathrm{GL}}_{2}^{+}(\RN)$ such that 
$$
\sigma(\alpha, \beta)=\sigma(\alpha_{0}, -\frac{1}{2}) \cdot \tilde{g}.
$$
\end{lem}

For a fixed $0<\alpha_{0}<\frac{1}{2}$, 
we denote by 
$$
\mathcal{K}:=  \sigma(\alpha_{0}, -\frac{1}{2}) \cdot \tilde{\mathrm{GL}}^{+}_{2}(\RN) \subset \Stab(\Ku(Q))
$$ 
the $\tilde{\mathrm{GL}}^{+}_{2}(\RN)$-orbit in $\Stab(\Ku(Q))$.
It is an open subset of $\Stab(\Ku(Q))$.
In particular, by Lemma \ref{GL-one-orbitV+}, these stability conditions $\sigma(\alpha,\beta)\in \mathcal{K}$ for all $(\alpha,\beta)\in V$.


\section{Proof of Theorem \ref{main-thm-nonempty}}

This section is devoted to the proof of Theorem \ref{main-thm-nonempty}.

\subsection{Slope stability of projection sheaves}
Let $Q$ be the smooth quadric threefold in $\PB^{4}$.
For a closed point $x\in Q$, 
there is a resolution of the skyscraper sheaf $\CO_{x}$ (cf. \cite[\S3, (3)]{Sch14}):
\begin{equation}\label{resol-sky-Q}
\xymatrix@C=0.5cm{
0 \ar[r]^{} & \CO(-H)\ar[r]^{} & S^{\oplus 2} \ar[r]^{} &  \CO^{\oplus 4} \ar[r]^{} & \CO(H) \ar[r]^{} & \CO_{x} \ar[r]^{} & 0.}
\end{equation}
Note that the spinor bundle in \cite{Sch14} is $S(H)$.
By the ideal sheaf exact sequence for $x\in Q$, 
we obtain a short exact sequence
\begin{equation}\label{structure-exact-Q}
\xymatrix@C=0.5cm{
0 \ar[r]^{} & I_{x}(H) \ar[r]^{} & \CO(H)  \ar[r]^{} & \CO_{x} \ar[r]^{} & 0,}
\end{equation}
where $I_{x}$ is the ideal sheaf of $x$ in $Q$.
Recall that $\mathcal{P}_{x}:=\PR(\CO_{x})\in \Ku(Q)$.
By the definition of the projection functor, we have 
\begin{equation}\label{structure-exact-cor-Q}
\xymatrix@C=0.5cm{
0 \ar[r]^{} & \mathcal{P}_{x} \ar[r]^{} & \CO^{\oplus 4} \ar[r]^{} & I_{x}(H) \ar[r]^{} & 0,}
\end{equation}
which is indeed obtained by splitting the exact sequence \eqref{resol-sky-Q}.
It follows that the projection sheaf  $\mathcal{P}_{x}$ is a torsion-free sheaf of rank $3$.
Moreover, from \eqref{resol-sky-Q}, there exists a short exact sequence
\begin{equation}\label{structure-exact-cor2-Q}
\xymatrix@C=0.5cm{
0 \ar[r]^{} & \CO(-H) \ar[r]^{} & S^{\oplus 2} \ar[r]^{} & \mathcal{P}_{x} \ar[r]^{} & 0.}
\end{equation}
From the exact sequence \eqref{structure-exact-cor-Q} (or \eqref{structure-exact-cor2-Q}), 
the Chern character of $\mathcal{P}_{x}$ is
$$
\ch(\mathcal{P}_{x})=(3,-H,-\frac{1}{2}H^{2},\frac{1}{3}H^{3}).
$$

By the construction of the projection sheaves, we have the following:

\begin{lem}\label{Px-ext}
$\Hom(\mathcal{P}_{x},\mathcal{P}_{x})=\CN$,  
$\Ext^{1}(\mathcal{P}_{x},\mathcal{P}_{x})=\CN^{4}$ 
and $\Ext^{i}(\mathcal{P}_{x},\mathcal{P}_{x})=0$ for $i\neq 0,1$.
\end{lem}

\begin{proof}
Since $\mathcal{P}_{x}\in \Ku(Q)$, by \eqref{structure-exact-cor-Q} and \eqref{structure-exact-Q}, we have
$$
\RHom(\mathcal{P}_{x},\mathcal{P}_{x})
\cong 
\RHom(I_{x}(H),\mathcal{P}_{x})[1]
\cong 
\RHom(\CO_{x},\mathcal{P}_{x})[2].
$$
Thus $\Ext^{i}(\CO_{x},\mathcal{P}_{x})=0$ for $i=0,1$.
By Grothendieck--Verdier duality (cf. \cite[Theorem 3.34]{Huy06}), 
we have $\RCH(\CO_{x},L) \cong \CO_{x}[-3]$ for all line bundle $L$ on $Q$.
In fact, for all line bundle $L$, we have
\begin{eqnarray*}
\RCH(\CO_{x},L)  
& = & \RCH(i_{\ast} \CN(x),L)  \\
& \cong & i_{\ast} \RCH_{x}(\CN(x), i^{!}L) \; (\textrm{by Grothendieck--Verdier duality}) \\
& \cong & i_{\ast} \RCH_{x}(\CN(x), \CN(x)[-3])=\CO_{x}[-3],
\end{eqnarray*}
where $i: x \hookrightarrow Q$ is the inclusion.
In particular, we get $\RHom(\CO_{x},L) \cong \CN[-3]$.
Applying $\Hom(\CO_{x},-)$ to \eqref{structure-exact-Q}, 
we obtain $\RHom(\CO_{x},I_{x}(H))=\CN[-1]\oplus \CN^{3}[-2]\oplus \CN^{3}[-3]$.
Then, applying $\Hom(\CO_{x},-)$ to \eqref{structure-exact-cor-Q}, 
we get $\Ext^{2}(\CO_{x}, \mathcal{P}_{x})=\CN$. 
By Hirzebruch--Riemann-Roch theorem, 
we have $\chi(\CO_{x},\mathcal{P}_{x})=-3$ and thus $\Ext^{3}(\CO_{x}, \mathcal{P}_{x})=\CN^{4}$. 
As a result, we get what we wanted.
\end{proof}

The following result shows that the projection sheaf $\mathcal{P}_{x}$ is $\mu_{H}$-slope stable and thus $H$-Gieseker stable.

\begin{prop}\label{Kp-slope-stable}
The sheaf $\mathcal{P}_{x}$ is reflexive, 
$\mu_{H}$-slope stable and locally free except at the point $x\in Q$.
\end{prop}

\begin{proof}
Note that $\ch_{\leq 1}(\mathcal{P}_{x})=(3,-H)$ is primitive.
Hence, to show that $\mathcal{P}_{x}$ is $\mu_{H}$-slope stable, 
it suffices to show that $\mathcal{P}_{x}$ is $\mu_{H}$-slope semistable.
We conclude this by contradiction. 
Suppose $\mathcal{P}_{x}$ is not $\mu_{H}$-slope semistable. 
Assume that $F \subset \mathcal{P}_{x}$ is the $\mu_{H}$-slope semistable subsheaf in the Harder--Narasimhan filtration of $\mathcal{P}_{x}$. 
Then we have $\mu_{H}(F)>\mu_{H}(\mathcal{P}_{x})=-\frac{1}{3}$ and the quotient $E_{D,V}/E$ is torsion-free. 
From the exact sequence \eqref{structure-exact-cor-Q}, 
one see that $F$ is also a subsheaf of $\CO^{\oplus 4}$ and thus $\mu_{H}(F)=0$. 
Set $\ch(F):=  (r,0,dH^{2},eH^{3})$. 
Then Chern character of the quotient sheaf $\CO^{\oplus 4}/F$ is 
$\ch(\CO^{\oplus 4}/F)=(4-r,0,-dH^{2},-eH^{3})$. 
By \eqref{structure-exact-cor-Q} again, 
since $I_{x}(H)$ is torsion-free, 
so is $\CO^{\oplus 4}/F$.
Since $\mu_{H}(\CO^{\oplus 4}/F)=0$ and $\CO^{\oplus 4}$ is $\mu_{H}$-slope semistable, 
so $\CO^{\oplus 4}/F$ is also $\mu_{H}$-slope semistable.
As a result, the classical Bogomolov inequality implies that $-rd \geq 0$ and $ (4-r) d\geq 0$.
Hence, we have $d = 0$. 
By applying Proposition \ref{sum-lines-detect} (i) to $F$ and $\CO^{\oplus 4}/F$, 
we obtain $e = 0$ and hence $F \cong  \CO^{\oplus r}$. 
This means $\mathcal{P}_{x}$ has global sections.
However, $H^{0}(Q,\mathcal{P}_{x})\cong  \Hom(\CO,\mathcal{P}_{x})=0$ as $\mathcal{P}_{x}\in \Ku(Q)$, a contradiction.

To prove that $\mathcal{P}_{x}$ is reflexive,
we first apply $\RCH(-,\CO)$ to \eqref{structure-exact-cor-Q} and obtain 
$$
\CExt^{i}(\mathcal{P}_{x},\CO)\cong \CExt^{i+1}(I_{x}(H),\CO)
$$
for any $i\geq 1$.
By Grothendieck--Verdier duality, 
it follows that $\RCH(\CO_{x},\CO) \cong \CO_{x}[-3]$.
Applying $\RCH(-,\CO)$ to \eqref{structure-exact-Q}, 
we get $\CExt^{j}(I_{x}(H),\CO)=0$ for any $j\geq 3$,
and $\CExt^{2}(I_{x}(H),\CO)\cong \CO_{x}$.
Hence $\CExt^{i}(\mathcal{P}_{x},\CO)=0$ for any $i\geq 2$,  
and  $\CExt^{1}(\mathcal{P}_{x},\CO)\cong \CO_{x}$ is supported in $x$.
This implies that $\mathcal{P}_{x}$ is reflexive, but not locally free at $x$.
\end{proof}

In particular,
the following result is an immediately consequence of Proposition \ref{Kp-slope-stable}, Lemma \ref{Px-ext} and \cite[Corollary 4.5.2]{HL10}.
Recall that the Chern character $\mathbf{v}=(3,-H,-\frac{1}{2}H^{2},\frac{1}{3}H^{3})$.

\begin{cor}
The moduli space $\overline{M}_{Q}(\mathbf{v})$ is smooth at $\mathcal{P}_{x}$ for every closed point $x\in Q$.
\end{cor}


\subsection{Bridgeland stability of projection sheaves}

Note that for any $\beta\in \RN$, the twisted Chern character $\ch_{\leq 2}^{\beta}$ of $\mathcal{P}_{x}$ is 
$$
\ch_{\leq 2}^{\beta}(\mathcal{P}_{x})=(3, -(3\beta+1)H,(\frac{3}{2}\beta^{2}+\beta-\frac{1}{2})H^{2}).
$$
Since $\ch_{0}(\mathcal{P}_{x})=3\neq 0$,
by Theorem \ref{tilt-wall-struct-thm},
there exists a unique numerical vertical wall for $\ch(\mathcal{P}_{x})$ given by $\beta=-\frac{1}{3}$, 
and the numerical semicircular walls are two sets of nested semicircles whose top points lie on the hyperbola
\begin{equation}\label{Kp-wall-apexes-hyperbola}
(\beta+\frac{1}{3})^{2}-\alpha^{2}=(\frac{2}{3})^{2}.
\end{equation}

We consider the following two subsets of $\widetilde{V}$:
\begin{eqnarray*}
\widetilde{V}_{L} &=& \{(\alpha, \beta)\in \widetilde{V} \mid  -1\leq \beta <-\frac{1}{3}, \alpha<-\beta; \; \textrm{or}\; -2< \beta < -1, \alpha\leq 2+\beta\}. \\
\widetilde{V}_{R} &=& \{(\alpha, \beta)\in \widetilde{V}\mid  -\frac{1}{3} \leq \beta < 0, \alpha<-\beta\}.
\end{eqnarray*}
Then $\widetilde{V}_{L}\cap \widetilde{V}_{R}=\emptyset$ and $\widetilde{V}=\widetilde{V}_{L}\cup \widetilde{V}_{R}$.
Let $\sigma(\alpha,\beta)$ be the stability condition in Proposition \ref{induced-stability-KuQ} for any $(\alpha,\beta)\in \widetilde{V}_{L}$.
Let $[\mathcal{P}_{x}]:=  2\lambda_{2}-\lambda_{1}$ be the numerical class of $\mathcal{P}_{x}$ in $\mathcal{N}(\Ku(Q))$.
The main purpose of this section is to show the following:

\begin{thm}[{Theorem \ref{main-thm-nonempty}}]
For any $(\alpha,\beta)\in \widetilde{V}$,
the moduli space $M_{\sigma(\alpha,\beta)}(2\lambda_{2}-\lambda_{1})$ is non-empty.
In particular, the moduli space $M_{\sigma}(2\lambda_{2}-\lambda_{1})$ is non-empty for all $\sigma\in \mathcal{K}$. 
\end{thm}

\begin{proof}
The first statement follows immediately from the Proposition \ref{start-point-stable} below.
The second statement is a consequence of the first one and Lemma  \ref{GL-one-orbitV+}.
\end{proof}

The following proposition is the core for the proof of Theorem \ref{main-thm-nonempty}.

\begin{lem}\label{no-wall-1}
There are no walls for tilt semistable objects $E$ with Chern character $\ch_{\leq 2}(E) = (3, -H, -\frac{1}{2} H^{2})$ to the left of the vertical wall $\beta=-\frac{1}{3}$.
\end{lem}

\begin{proof}
Since the apexes of the numerical semicircular walls for $E$ lie on the hyperbola \eqref{Kp-wall-apexes-hyperbola}, hence every numerical semicircular wall to the left of the vertical wall intersects with the vertical line $\beta=-1$.
As a result, it is sufficient to show that there is no an actual wall for $E$ along the vertical line $\beta=-1$.
Hence, we may assume that $E$ is a $\sigma_{\alpha,-1}$-semistable object with Chern character $\ch_{\leq 2}(E) = (3, -H, -\frac{1}{2} H^{2})$.
In the following, we show that there are no actual walls along $\beta = -1$ for $E$ with respect to $\sigma_{\alpha,-1}$.

Note that $\ch_{\leq 2}^{-1}(E)=(3,2H,0)$.
By Definition \ref{wall-def}, 
an actual wall would be given by a short exact sequence of $\sigma_{\alpha,-1}$-semistable objects in the heart $\Coh^{-1}(Q)$
\begin{equation}\label{sky-point-destab-sequ}
\xymatrix@C=0.5cm{
0 \ar[r]^{} & A \ar[r]^{} & E \ar[r]^{} & B \ar[r]^{} & 0}
\end{equation}
such that the following conditions hold (see Remark \ref{actual-restrict}):
\begin{enumerate}
\item[(1)] $\mu_{\alpha, -1}(A)=\mu_{\alpha, -1}(E)=\mu_{\alpha, -1}(B)$;
\item[(2)] $\Delta_{H}(A)\geq 0$, $\Delta_{H}(B)\geq 0$;
\item[(3)]  $\Delta_{H}(A) \leq \Delta_{H}(E), \Delta_{H}(B) \leq \Delta_{H}(E)$.
\end{enumerate}
By \eqref{sky-point-destab-sequ},
the truncated twisted Chern characters $\ch^{-1}_{\leq 2}$ of $A$ and $B$ satisfy
\begin{equation}\label{sky-point-destab-equation}
(3,2H,0)
=(a,bH, \frac{c}{2}H^{2})+(3-a,(2-b)H, \frac{1-c}{2}H^{2})
\end{equation}
for some $a, b, c \in \Z$. 
Since $A, B\in \Coh^{-1}(Q)$,  
so $b\geq 0$ and $2-b\geq 0$.
By $(1)$, $A$ and $B$ could not have infinity tilt slopes and it follows that $b=1$. 
By changing the roles of $A$ and $B$, we may assume $a\geq 2$.
By $(1)$, $(2)$ and $(3)$, we obtain the following restrictions:
\begin{enumerate}
  \item[$(i)$] $(2a-3)\alpha^{2}=2c$;
  \item[$(ii)$] $-3 \leq ac\leq 1$.
\end{enumerate}
By $(i)$ and $a\geq 2$, we get $c\geq 1$.
This contradicts to $(ii)$. 
As a consequence, there is no a solution for the equation \eqref{sky-point-destab-equation} and thus concludes the proof. 
\end{proof}

\begin{prop}\label{start-point-stable}
\begin{enumerate}
  \item[(1)] The sheaf $\mathcal{P}_{x}$ is $\sigma(\alpha,\beta)$-stable for any $(\alpha,\beta)\in \widetilde{V}_{L}$.
  \item[(2)] The sheaf $\mathcal{P}_{x}[1]$ is $\sigma(\alpha,\beta)$-stable for any $(\alpha,\beta)\in \widetilde{V}_{R}$.
\end{enumerate}
\end{prop}

\begin{proof}
(1)  For each $(\alpha,\beta)\in \widetilde{V}_{L}$, 
we have $\beta<-\frac{1}{3}=\mu_{H}(\mathcal{P}_{x})$.
By Proposition \ref{Kp-slope-stable}, 
$\mathcal{P}_{x}$ is $\mu_{H}$-slope stable
and thus $\mathcal{P}_{x} \in \Coh^{\beta}(Q)$.
Since $H^{2}\cdot \ch_{1}^{\beta}(\mathcal{P}_{x})=-2(3\beta+1)>0$,
by Proposition \ref{limit-tilt-stab-Gie-stab-prop}, it follows that 
$\mathcal{P}_{x}$ is $\sigma_{\alpha, \beta}$-stable for all $\alpha$ sufficiently large.
Next we need to show that there are no actual walls for the stability of $\mathcal{P}_{x}$ with respect to $\sigma_{\alpha,\beta}$.
In fact, this is a consequence of Theorem \ref{tilt-wall-struct-thm} and Lemma \ref{no-wall-1}.
Since $\sigma^{0}_{\alpha,\beta}$ is just a rotation of $\sigma_{\alpha,\beta}$,
so $\mathcal{P}_{x}$ is $\sigma^{0}_{\alpha,\beta}$-stable and hence it is $\sigma(\alpha,\beta)$-stable.

(2) Since $\mathcal{P}_{x}$ is reflexive and $\mu_{H}$-slope stable and $\mu_{H}(\mathcal{P}_{x})=-\frac{1}{3}$, $\mathcal{P}_{x}\in \Coh^{-\frac{1}{3}}(Q)$.
By \cite[Proposition 4.18]{BBF+}, 
$\mathcal{P}_{x}[1]$ is $\sigma_{\alpha,-\frac{1}{3}}$-stable for all $\alpha>0$.
Hence, $\mathcal{P}_{x}[1]$ is $\sigma_{\alpha,-\frac{1}{3}}$-stable for $\alpha \gg 0$.
Similar to the proof of Lemma \ref{no-wall-1},
there are no actual walls along $\beta =0 $ for $\mathcal{P}_{x}[1]$ with respect to $\sigma_{\alpha,0}$.
Analogous to (1), $\mathcal{P}_{x}[1]$ is $\sigma(\alpha,\beta)$-stable for any $(\alpha,\beta)\in \widetilde{V}_{R}$. Alternatively, since $\widetilde{V}_{R}\subset V$, the second statement follows immediately from Lemma \ref{GL-one-orbitV+} and the first assertion, as the stability is preserved under the $\tilde{\mathrm{GL}}^{+}_{2}(\RN)$-action.
\end{proof}


\section{Construction of sheaves}

Let $\overline{M}_{Q}(\mathbf{v})$ be the moduli space of $H$-Gieseker semistable sheaves on $Q$ with Chern character 
$$
\mathbf{v}:=  \ch(\mathcal{P}_{x})=(3,-H,-\frac{1}{2}H^{2},\frac{1}{3}H^{3}).
$$
We denote by $M_{Q}(\mathbf{v})\subset \overline{M}_{Q}(\mathbf{v})$ the open locus of $H$-Gieseker semistable vector bundles on $Q$ with Chern character  $\mathbf{v}$.
In Proposition \ref{Kp-slope-stable}, it is proved that the sheaf $\mathcal{P}_{x}$ is a $\mu_{H}$-slope stable reflexive sheaf  and hence a $H$-Gieseker semi-stable sheaf. 
Conversely, 
since $\ch_{\leq 1}(\mathcal{P}_{x})$ is primitive, 
any $H$-Gieseker semistable sheaf with Chern character $\ch(\mathcal{P}_{x})$ is $\mu_{H}$-slope stable.

To analyze the moduli space $\overline{M}_{Q}(\mathbf{v})$, 
we need to characterize the $H$-Gieseker semistable sheaves on the smooth quadric threefold  $Q$ with Chern character $\mathbf{v}$.
The next two sections are devoted to this purpose.
It is well-known that any hyperplane section $Y= H\cap Q \subset \PB^{4}$ is either smooth quadric surface or a quadric cone.

Following the construction of sheaves in \cite[Section 5]{BBF+},
we suppose that $Y\subset Q$ is a hyperplane section. 
Let $D$ be an effective Weil divisor and $V\subset H^{0}(Y,\CO_{Y}(D))$ a non-zero subspace.
Then there is a natural evaluation morphism $\mathrm{ev}: V\otimes \CO_{Y}
\rightarrow \CO_{Y}(D)$.
Let $\mathcal{E}_{D,V}$ be the cone of the induced morphism  $\overline{\mathrm{ev}}: V\otimes \CO\rightarrow \CO_{Y}(D)$ in $\D(Q)$.
Then there exists a distinguished triangle 
$$
\xymatrix@C=0.5cm{ V\otimes \CO \ar[r]^{\overline{\mathrm{ev}}} & \CO_{Y}(D)  \ar[r]^{} & \mathcal{E}_{D,V} }. 
$$
Consider the long exact sequence of cohomological sheaves, 
there is a long exact sequence of sheaves
\begin{equation}\label{constrct-sheaf-exact-sequ}
\xymatrix@C=0.5cm{
0 \ar[r]^{} & \CH^{-1}(\mathcal{E}_{D,V}) 
\ar[r]^{} & V \otimes \CO
 \ar[r]^{\overline{\mathrm{ev}}} & \CO_{Y}(D)
  \ar[r]^{} & \CH^{0}(\mathcal{E}_{D,V}) 
 \ar[r]^{} & 0.}
\end{equation}
For simplicity, we denote by $E_{D,V}:=  \CH^{-1}(\mathcal{E}_{D,V})$.
If $V=H^{0}(Y,\CO_{Y}(D))$, we set $\mathcal{E}_{D}:=  \mathcal{E}_{D,V}$ and $E_{D}:=  E_{D,V}$.
Then one has:

\begin{prop}\label{Construct-sh-prop}
The sheaf $E_{D,V}$ is reflexive and $\mu_{H}$-slope stable.
Moreover, if $\CH^{0}(\mathcal{E}_{D,V})=0$, 
then $E_{D,V}$ is locally free.
\end{prop}

\begin{proof}
Based on Proposition \ref{sum-lines-detect}, 
the proof is the same as that of \cite[Lemma 5.1]{BBF+}.
\end{proof}

Using this construction, 
the following proposition essentially provides all $\mu_{H}$-slope stable vector bundle of rank $3$ with Chern character $\mathbf{v}$ (see Theorem \ref{claasifying-sh-ch-v}).

Suppose that $Y\in |H|$ is a smooth hyperplane section of $Q$.
Let $D$ be a divisor of type $(2,0)$ on $Y$ (i.e., $D$ is either a disjoint sum of two lines or a double line in $Y$).
Then we have:

\begin{prop}\label{new-G-st-sh}
The sheaf $E_{D}$ is a $\mu_{H}$-slope stable vector bundle of rank $3$ with Chern character $\ch(E_{D})=\mathbf{v}$.
\end{prop}

\begin{proof}
Since $D$ is a divisor of type $(2,0)$ on the smooth hyperplane section $Y$,
so $H^{0}(Y,\CO_{Y}(D))\cong H^{0}(\PB^{1},\CO_{\PB^{1}}(2))\cong \CN^{3}$.
By the hypothesis, $\CO_{Y}(D)$ is globally generated. 
As a result, the morphism $\overline{\mathrm{ev}}$  in \eqref{constrct-sheaf-exact-sequ} is surjective, 
and thus there is a short exact sequence
$$
\xymatrix@C=0.5cm{
0 \ar[r]^{} & E_{D}
\ar[r]^{} & \CO^{\oplus 3}
 \ar[r]^{} & \CO_{Y}(D)
 \ar[r]^{} & 0.}
$$

By Proposition \ref{Construct-sh-prop}, $E_{D}$ is a $\mu_{H}$-slope stable vector bundle of rank $3$.
In the rest of the proof, we will compute the Chern characters of $\CO_{Y}(D)$ and $E_{D}$.
Since $D$ is either a disjoint sum of two lines or a double line, 
hence the Chern character $\ch(\CO_{D})=2\ch(\CO_{\ell})$ for a line $\ell\subset Y$.
For a line $\ell\subset Q$, by Grothendieck--Riemann--Roch theorem, 
we have $c_{i}(\CO_{\ell})=0$ for $i<2$ and $c_{2}(\CO_{\ell})=-\frac{1}{2}H^{2}$.
Hence, the Chern character of $\CO_{\ell}$ is $\ch(\CO_{\ell})=(0,0,\frac{1}{2}H^{2},\ch_{3}(\CO_{\ell}))$.
By Hirzebruch--Riemann--Roch formula \eqref{HRR-quadric3},
we get $\chi(\CO_{\ell})=1$ and thus $\ch_{3}(\CO_{\ell})=-\frac{1}{4}H^{3}$.
Therefore, the Chern character $\ch(\CO_{D})=(0,0,H^{2},-\frac{1}{2}H^{3})$.
Consider the ideal sheaf exact sequence
$$
\xymatrix@C=0.5cm{
0 \ar[r]^{} & \CO_{Y}(-D) \ar[r]^{} & \CO_{Y} \ar[r]^{} & \CO_{D} \ar[r]^{} & 0.}
$$
By Grothendieck--Verdier duality,
the derived dual 
$$\RCH(\CO_{D},\CO_{Y}) \cong \RCH(\CO_{D},\CO) \otimes \CO(-H)[1].
$$
Since $\ch(\CO_{Y})=\ch(\CO)-\ch(\CO(-H))=(0,H,-\frac{1}{2}H^{2},\frac{1}{6}H^{3})$, 
then we have 
$$
\ch(\CO_{Y}(D))=\ch(\CO_{Y})-\ch(\CO_{D}^{\vee}\otimes\CO(-H)[1])=(0,H,\frac{1}{2}H^{2},-\frac{1}{3}H^{3}),
$$
where $\CO_{D}^{\vee}=R\mathcal{H}om(\CO_{D},\CO)$ is the derived dual.
As a consequence, the Chern character $\ch(E_{D})=3\ch(\CO)-\ch(\CO_{Y}(D))=\mathbf{v}$ is what we wanted.
This completes the proof.
\end{proof}

\begin{rem}\label{projection-construct-VB}
Recall that $M_{Q}(-1,2)$ is the moduli space of stable vector bundles of rank $2$ with Chern classes $c_{1}=-1$ and $c_{2}=2$ on $Q$. 
For every vector bundle $F\in M_{Q}(-1,2)$, the Chern character of $F$ is 
$
\ch(F)=(2,-H,-\frac{1}{2}H^{2},\frac{1}{3}H^{3}).
$
By \cite[Theorem 4.1]{OS94}, the moduli space $M_{Q}(-1,2)$ is rational, irreducible and smooth of dimension $6$.
Moreover, by \cite[Proposition 4.2]{OS94}, 
$H^{j}(Q,F(-H))=0$ for all $j\in \ZN$, $H^{1}(Q,F)=\CN$ and $H^{i}(Q,F)=0$ for $i\neq 1$. 
Therefore, $F$ belongs to the right orthogonal complement of $\CO(H)$.
Note that for any $F\in M_{Q}(-1,2)$, there is a distinguished triangle  
$
\xymatrix@C=0.3cm{
\CO[-1] \ar[r]^{} & F \ar[r]^{} & \mathrm{L}_{\CO}F} 
$
in $\D(Q)$.
Set $E_{F}:=\mathrm{L}_{\CO}F\in \Ku(Q)$.
Since $\Ext^{1}(\CO,F)\cong H^{1}(Q,F) =\CN$, 
hence there is a unique non-trivial extension 
\begin{equation}\label{VB-construct-two}
\xymatrix@C=0.5cm{
0 \ar[r]^{} & F \ar[r]^{} & E_{F} \ar[r]^{} &\CO \ar[r]^{} & 0.} 
\end{equation}
It follows that the sheaf $E_{F}$ is a vector bundle of rank $3$ with Chern character  $\ch(E_{F})=\mathbf{v}$.
By the lemma below, 
this provides another construction of $\mu_{H}$-slope stable rank $3$ vector bundles with Chern character $\mathbf{v}$.
\end{rem}

\begin{lem}
The vector bundle $E_{F}$ is $\mu_{H}$-slope stable for every vector bundle $F\in M_{Q}(-1,2)$.
\end{lem}

\begin{proof}
By Hoppe's criterion (see \cite[Lemma 2.6]{Hop84}), 
if $H^{0}(Q, E_{F})=0=H^{0}(Q,\wedge^{2} E_{F})$,
then $E_{F}$ is $\mu_{H}$-slope stable. 
Suppose $H^{0}(Q,E_{F})\neq 0$. 
For a non-zero morphism $\varphi\in \Hom(\CO,E_{F})$, 
it must be either factor through $F$ or split $E_{F}$, both are impossible, 
as $\mu_{H}(F)=-\frac{1}{2}$ and the exact sequence \eqref{VB-construct-two} is non-split.
Considering the wedge of \eqref{VB-construct-two}, there is a short exact sequence
\begin{equation}\label{wedge-VB-construct-two}
\xymatrix@C=0.5cm{
0 \ar[r]^{} & \wedge^{2} F \ar[r]^{} & \wedge^{2} E_{F} \ar[r]^{} & F \ar[r]^{} & 0.} 
\end{equation}
Since $\wedge^{2} F$ and $F$ are line bundles, so they are $\mu_{H}$-slope stable. 
Both $\mu_{H}$-slopes of $\wedge^{2} F$ and $F$ are less than $0$,
it follows from \eqref{wedge-VB-construct-two} that $H^{0}(Q,\wedge^{2} E_{F})=0$.
This finishes the proof.
\end{proof}

To describe $H$-Gieseker semistable sheaves with Chern character $\mathbf{v}$, 
we need control the maximal Chern characters $\ch_{2}$ and $\ch_{3}$ of sheaves with Chern character $\ch_{\leq 1}=(3,-1)$.

\begin{prop}\label{maximal-ch-Q}
Suppose that  $E$ is a $\mu_{H}$-slope stable sheaf with Chern character 
$$
\ch(E)=(3,-H, \ch_{2}(E), \ch_{3}(E)).
$$
Then $H \cdot \ch_{2}(E) \leq -\frac{1}{2} H^{3}$. 
If furthermore $H\cdot \ch_{2}(E)=-\frac{1}{2} H^{3}$, then $\ch_{3}(E) \leq \frac{1}{3} H^{3}$. 
In particular, every $\mu_{H}$-slope stable sheaf with Chern character $\mathbf{v}$ is reflexive.
\end{prop}

\begin{proof}
 We follow the proof of \cite[Proposition 6.6]{BBF+}.
By hypothesis, $E$ is $\mu_{H}$-slope stable, 
then the classical Bogomolov inequality yields $H \cdot \ch_{2}(E) \leq \frac{1}{6} H^{3}$. 
By Lemma \ref{Num-Chow-Q},   
we know that $\ch_{2}(E)\in \frac{H^{2}}{2} \ZN$ and thus $H \cdot \ch_{2}(E)\leq 0$.
If $H \cdot \ch_{2}(E)=0$, 
it follows from \cite[Proposition 3.2]{Li19} that $\ch_{0}(E)=1$ or $2$, 
a contradiction with $\ch_{0}(E)=3$.

For the second assertion, we assume $H \cdot \ch_{2}=-\frac{1}{2}H^{3}$. 
Since $E$ is $\mu_{H}$-slope stable and $\mu_{H}(E)=-\frac{1}{3}$,
by Proposition \ref{limit-tilt-stab-Gie-stab-prop}, 
$E \in \Coh^{\beta}(X)$ is $\sigma_{\alpha, \beta}$-stable for $\alpha \gg 0$ and $\beta < -\frac{1}{3}$. 
Hence, $E$ is $\sigma_{\alpha,-1}$-stable for $\alpha \gg 0$. 
By Lemma \ref{no-wall-1}, 
it follows that $E$ is $\sigma_{\alpha,-1}$-stable for any $\alpha > 0$. 
Note that 
$
\displaystyle \lim_{\alpha\to 0^{+}}\mu_{\alpha, -1}(\CO(-3H)[1])=-1< 0=\lim_{\alpha\to 0^{+}} \mu_{\alpha, -1}(E).
$
Together with Serre duality, it implies 
\begin{equation}\label{H2-shEV=0}
H^{2}(Q, E)\cong \Hom(E,\CO(-3H)[1])=0.
\end{equation}
Since $\mu_{H}(\CO(-3H))=-3<\mu_{H}(E)=-\frac{1}{3}< 0=\mu(\CO_X)$, 
the slope stability implies $\hom(\CO,E) = 0$ and $\hom(\CO,E[3])\cong \Hom(E,\CO(-3H))=0$. 
Therefore, 
by Hirzebruch--Riemann--Roch formula \eqref{HRR-quadric3},
we have 
$$
\ch_3(E)-\frac{1}{3}H^{3} =\chi(E)= -\hom(\CO,E[1]) \leq 0.
$$
It follows that the second assertion holds.

Finally, we show that a $\mu_{H}$-slope stable sheaf $E$ with Chern character $\mathbf{v}$ is reflexive. 
Since $E$ is torsion-free, there is a short exact sequence
$$
\xymatrix@C=0.5cm{
0 \ar[r]^{} & E \ar[r]^{} & E^{\vee\vee} \ar[r]^{} & T \ar[r]^{} & 0,}
$$
where $T$ is a torsion sheaf supported in codimension $\geq 2$.
We write $\ch(T)=(0,0,\frac{s}{2}H^{2}, \frac{t}{12} H^{3})$, where $s\geq 0$ and $t\in \ZN$.
Since the dual of a $\mu_{H}$-slope stable sheaf is also $\mu_{H}$-slope stable, 
so $E^{\vee \vee}$ is $\mu_{H}$-slope stable. 
Hence, we have $ H\cdot \ch_{2}(E^{\vee\vee})=H \cdot \ch_{2}(E)+\frac{s}{2}H^{3} \leq -\frac{1}{2}H^{3}=H \cdot \ch_{2}(E)$.
Hence, we have $s=0$. 
This means that $T$ is a torsion sheaf supported in codimension $3$ and thus $t\geq 0$.
Again, we have
$\ch_{3}(E^{\vee\vee})=\ch_{3}(E)+\frac{t}{12}H^{3} \leq \frac{1}{3}H^{3}=\ch_{3}(E)$.
Therefore, we obtain $\ch(E)=\ch(E^{\vee \vee})$ and thus $T=0$.
This finishes the proof of Proposition \ref{maximal-ch-Q}.
\end{proof}

As a direct consequence, the following version of large volume limit holds.

\begin{cor}\label{large-limit-v}
If $\beta>-\frac{1}{3}$,
then an object $\tilde{E} \in \Coh^{\beta}(Q)$ of Chern character $\ch(\tilde{E})=-\mathbf{v}$ 
is $\sigma_{\alpha,\beta}$-semistable for $\alpha \gg 0$ if and only if $\tilde{E} \cong E[1]$ for a $\mu_{H}$-slope stable reflexive sheaf $E$.
\end{cor}

\begin{proof}
We follow the proof of \cite[Lemma 6.7]{BBF+}.
If $E$ is a $\mu_{H}$-slope stable reflexive sheaf with $\ch(E)=\mathbf{v}$,
it follows immediately from \cite[Proposition 4.18]{BBF+} that $E[1]$ is $\sigma_{\alpha,\beta}$-semistable for $\alpha \gg 0$ and $\beta>-\frac{1}{3}$.

Conversely, suppose that $\tilde{E} \in \Coh^{\beta}(Q)$ is $\sigma_{\alpha,\beta}$-semistable for $\alpha \gg 0$.
Since $\beta>-\frac{1}{3}$, 
by \cite[Proposition 4.9]{BBF+}, 
the cohomological sheaf $\mathcal{H}^{-1}(\tilde{E})$ is a $\mu_{H}$-slope stable reflexive sheaf and $\mathcal{H}^{0}(\tilde{E})$ is a torsion sheaf supported in codimension $\geq 2$.
As a result, the Chern character 
$$
\ch(\mathcal{H}^{-1}(\tilde{E}))=(3,-H,-\frac{1}{2}H^{2}+\ch_{2}(\mathcal{H}^{0}(\tilde{E})),\frac{1}{3}H^{3}+\ch_{3}(\mathcal{H}^{0}(\tilde{E}))).
$$
It is a consequence of Proposition \ref{maximal-ch-Q} that $\ch_{2}(\mathcal{H}^{0}(\tilde{E}))=0$ and  $\ch_{3}(\mathcal{H}^{0}(\tilde{E}))=0$.
Hence, $\mathcal{H}^{0}(\tilde{E})=0$ and $\tilde{E}\cong \mathcal{H}^{-1}(\tilde{E})[1]$.
\end{proof}


\section{Classification of sheaves}

The main result in this section is the following classification of $H$-Gieseker semistable sheaves with Chern character $\mathbf{v}$.

\begin{thm}\label{claasifying-sh-ch-v}
A sheaf $E\in \Coh(Q)$ with Chern character $\mathbf{v}$ is $H$-Gieseker semistable if and only if it is exactly one of the following:
  \begin{enumerate}
  \item[(1)] the reflexive sheaf $\mathcal{P}_{x}$ for a closed point $x \in Q$;
  \item[(2)]  the vector bundle $E_{D}$ for a divisor $D$ of type $(2,0)$ in a smooth hyperplane section $Y \subset Q$.
\end{enumerate}
\end{thm}

This theorem is greatly inspired by \cite[Theorem 6.1]{BBF+} and we will prove it by using the same idea as that for \cite[Theorem 6.1]{BBF+}.

Suppose that $G\in \D(Q)$ is an object of Chern character $\ch_{\leq 2}(G) = (0, H, \frac{1}{2}H^2)$.
Since $\ch_{0}(G)=0$, by Theorem \ref{tilt-wall-struct-thm},
every numerical wall $W(r)$ for $\ch(G)$ is a semicircle with apex lies on the vertical line $\beta=\frac{1}{2}$ in the upper half plane:
\begin{equation}\label{torsion-sh-wall}
W(r)=\{ (\alpha,\beta)\in \RN_{>0}\times \RN \mid (\beta-\frac{1}{2})^{2}+\alpha^{2}=r^{2}, r>0\}.
\end{equation}

\begin{lem}\label{actual-wall-for-torsionsh}
The numerical wall $W(\frac{1}{2})$ is the only possible actual wall for $\ch_{\leq 2}(G) = (0, H, \frac{1}{2}H^2)$. 
Moreover, if $G$ is strictly semistable along the wall $W(\frac{1}{2})$, 
then any Jordan--H\"{o}lder filtration of $G$ is given by either
$$
\xymatrix@C=0.3cm{
0 \ar[r]^{} & I_{Z}(H) \ar[r]^{} & G \ar[r]^{} &\CO[1] \ar[r]^{} & 0} \;
\textrm{or}\;
\xymatrix@C=0.3cm{
0 \ar[r]^{} & \CO[1] \ar[r]^{} & G \ar[r]^{} & I_{Z}(H) \ar[r]^{} & 0,} 
$$
where $Z \subset Q$ is a $0$-dimensional subscheme of length $\frac{1}{6}H^{3}-\ch_{3}(G)$.
\end{lem}

\begin{proof}
Suppose that the numerical wall $W(r)$ is an actual wall.
Then there is a destabilizing short exact sequence of $\sigma_{\alpha,\frac{1}{2}}$-semistable objects
$$
\xymatrix@C=0.5cm{
0 \ar[r]^{} &  A \ar[r]^{} &  G \ar[r]^{} &  B\ar[r]^{} &  0
}
$$ 
in $\Coh^{\frac{1}{2}}(Q)$.
Then the truncated twisted Chern characters satisfy
$$
(0,H,0)=(a,\frac{b}{2}H,\frac{c}{8}H^{2})+(-a,\frac{2-b}{2}H^{2},-\frac{c}{8}H^{2})
$$
where $a, b, c\in \ZN$.
To compute the wall, we need to solve the above equation.
Since $A, B \in \Coh^{\frac{1}{2}}(Q)$ and both have finite $\sigma_{\alpha,\frac{1}{2}}$-slope,
we obtain $b=1$. 
As $\ch_{1}(A)=\ch_{1}^{\frac{1}{2}}(A)+\frac{a}{2}H=\frac{1+a}{2}H$, 
so $a$ must be odd.
We may assume that $a\leq -1$, by changing the roles of $A$ and $B$.
Since $\mu_{\alpha,\frac{1}{2}}(A)=\mu_{\alpha,\frac{1}{2}}(G)$, 
then $4a\alpha^{2}=c$ and thus $c\leq -1$.
Since $0\leq \Delta_{H}(A)\leq \Delta_{H}(G)=(H^{3})^{2}$, 
we get $ac=1$.
This yields $(a,c)=(-1,-1)$ and thus $\alpha=\frac{1}{2}$.
Therefore, the unique possible wall is $W(\frac{1}{2})$.
Then we have $\ch_{\leq2}^{\frac{1}{2}}(A)=(-1,\frac{1}{2}H, -\frac{1}{8}H^{2})$. 
Hence, $\ch_{\leq2}(A)=(-1,0, 0)$ and by Proposition \ref{sum-lines-detect} (ii), 
we obtain $A\cong \CO[1]$ and thus $\ch(B)=(1,H, \frac{1}{2}H^{2}, \ch_{3}(G))$.
Since $B$ is $\sigma_{\frac{1}{2},\frac{1}{2}}$-semistable and $H^{2} \cdot \ch_{1}^{\frac{1}{2}}(B)=\frac{1}{2} H^{3}$,
it follows from the previous arguments that $B$ is $\sigma_{\alpha,\frac{1}{2}}$-semistable for $\alpha\gg0$.
Since $\mu_{H}(B)=1>\frac{1}{2}$, by Corollary \ref{limit-tilt-stab-Gie-stab}, 
$B$ is $\mu_{H}$-slope stable sheaf.
Since $\ch(B(-H))=(1,0, 0, \ch_{3}(G)-\frac{1}{6}H^{2})$, 
by Proposition \ref{sum-lines-detect} (i),
we have $\ch_{3}(G)-\frac{1}{6}H^{2}\leq 0$. 
Since $F:=  B(-H)$ is a torsion-free sheaf of rank $1$, so $F^{\vee\vee}$ is a line bundle with $c_{1}(F^{\vee\vee})=0$. Hence, $F^{\vee\vee}\cong \CO$, as $Q$ is of Picard rank $1$.
Therefore, $F\subset F^{\vee\vee}=\CO$ and thus $B(-H)\cong I_{Z}$ for a closed subscheme $Z$ of $Q$ of codimension $3$.
Here $Z$ is of length $\frac{1}{6}H^{3}-\ch_{3}(G)$.
By changing the roles of $A$ and $B$,  this completes the proof.
\end{proof}

The following result gives a control of the bound of the third Chern character $\ch_{3}$.

\begin{cor}\label{control-ch3}
Let $G$ be a $\sigma_{\alpha,\beta}$-semistable object with $\ch_{\leq2}(G)=(0,H,\frac{1}{2}H^{2})$.
Then $\ch_{3}(G)\leq \frac{1}{6}H^{3}$.
If moreover, $\ch_{3}(G)=\frac{1}{6}H^{3}$ and $(\alpha,\beta)$ above the wall $W(\frac{1}{2})$,
then $G\cong \CO_{Y}(H)$ for some $Y \in |H|$. 
\end{cor}

\begin{proof}
Suppose $\ch_{3}(G) \geq \frac{1}{6} H^3$. 
According to Proposition \ref{actual-wall-for-torsionsh}, 
the only possible wall is $W(\frac{1}{2})$. 
Therefore, $G$ has to be tilt semistable along the numerical wall $W(\frac{1}{2})$. 
By direct computations, 
the numerical wall $W(G,\CO(-2H)[1])=W(\frac{5}{2})$ 
and thus $W(\frac{1}{2})$ lies below $W(G,\CO(-2H)[1])$.
Hence, by Serre duality, one has $\ext^{2}(\CO(H),G)=\hom(G,\CO(-2H)[1])=0$. 
As a result, by Hirzebruch--Riemann--Roch formula \eqref{HRR-quadric3}, 
we have
$$
\hom(\CO(H), G) \geq \chi(\CO(H), G) = \ch_{3}(G)+\frac{1}{3}H^{3} > 0.
$$ 
As a result, $W(\frac{1}{2})$ is an actual wall for $G$.
By Proposition \ref{actual-wall-for-torsionsh} again, 
the destabilizing sequence for $G$ is the following exact sequence
$
\xymatrix@C=0.3cm{
0 \ar[r]^{} & \CO(H) \ar[r]^{} & G \ar[r]^{} &\CO[1] \ar[r]^{} & 0.} 
$
It follows that $G\cong \CO_{Y}(H)$ for some $Y \in |H|$ and thus $\ch_{3}(G) = \frac{1}{6} H^{3}$.
\end{proof}

Now we can state the classification result of torsion sheaves with Chern character 
$$
\ch=(0,H,\frac{1}{2}H^{2},-\frac{1}{3}H^{3}).
$$

\begin{prop}\label{classify-torsion-sheaf}
The wall $W(\frac{1}{2})$ is the unique actual wall in tilt stability for objects $G$ with Chern character $\ch(G)=(0,H,\frac{1}{2}H^{2},-\frac{1}{3}H^{3})$.
\begin{enumerate}
\item[(i)] Above the wall $W(\frac{1}{2})$ the moduli space of tilt semistable objects is the moduli space of $H$-Gieseker semistable sheaves, and contains exactly the following two types of sheaves $G$:
    \begin{enumerate}
        \item[(1)] $G = I_{x/Y}(H)$ for a hyperplane section $Y \subset G$ and a point $x \in Y$; and
        \item[(2)] $G = \CO_{Y}(D)$ where $D$ is a Weil divisor on some hyperplane section  $Y \subset G$.
    \end{enumerate}
\item[(ii)] Below the wall $W(\frac{1}{2})$ the moduli space of tilt semistable objects contains exactly the following two types of objects $G$:
    \begin{enumerate}
    \item[(a)] the unique non-trivial extensions
         $ \xymatrix@C=0.3cm{0 \ar[r]^{} & \CO[1] \ar[r]^{} & G \ar[r]^{} & I_{x}(H) \ar[r]^{} & 0} $
         for a point $x \in Q$; and
    \item[(b)] $G=\CO_{Y}(D)$ where $D$ is a Weil divisor on a hyperplane section $Y \subset G$.
    \end{enumerate}
\end{enumerate}
\end{prop}

\begin{proof}
Thanks to Lemma \ref{actual-wall-for-torsionsh} and Corollary \ref{control-ch3},
the proof of Proposition \ref{classify-torsion-sheaf} is indeed that of \cite[Proposition 6.2]{BBF+}.
Here two little differences needed to be clarified.
(i) For a closed point $x\in Q$, 
by Serre duality and the exact sequence \eqref{structure-exact-Q},
one has
$$
\Ext^{1}(I_{x}(H),\CO[1])\cong H^{1}(Q,I_{x}(-2H)) \cong H^{0}(Q,\CO_{x})=\CN.
$$ 
(ii) In our case, 
for the non-negative integer $t$ in the proof of \cite[Proposition 6.2]{BBF+}, 
by Corollary \ref{control-ch3},  one has $0\leq t\leq \frac{2}{3}$ and thus $t=0$.
\end{proof}

Now we are ready to prove Theorem \ref{claasifying-sh-ch-v}.
In fact, it is an immediately consequence of Proposition \ref{Kp-slope-stable},  
the below Lemma \ref{classify-tor-lem1} and  Lemma \ref{converse-to-new-G-st-sh}.

\begin{lem}\label{classify-tor-lem1}
The numerical wall $W(\frac{1}{2})$ is the unique actual wall for objects with Chern character $-\mathbf{v}=(-3,H,\frac{1}{2}H^{2},-\frac{1}{3}H^{3})$ to the right of the vertical wall $\beta=-\frac{1}{3}$, and there are no tilt semistable objects below the wall $W(\frac{1}{2})$. 
Moreover, any tilt semistable object $\tilde{E}$ with Chern character $-\mathbf{v}$ fits into one of the following:
\begin{enumerate}
    \item $\tilde{E}$ fits into a short exact sequence
    $\xymatrix@C=0.3cm{
    0 \ar[r]^{} & \CO_{Y}(D) \ar[r]^{} & \tilde{E} \ar[r]^{} & \CO^{\oplus 3}[1] \ar[r]^{} & 0}$,
    where $D$ is a Weil divisor on a hyperplane section $Y \in |H|$;
    \item $\tilde{E}$ fits into a short exact sequence
    $\xymatrix@C=0.3cm{
     0 \ar[r]^{} &  I_{x}(H) \ar[r]^{} & \tilde{E}\ar[r]^{} & \CO^{\oplus 4}[1] \ar[r]^{} & 0}$,
    where $x \in Q$ is a point.
\end{enumerate}
\end{lem}

\begin{proof}
The numerical walls for $-\mathbf{v}$ are the same as that for $\mathbf{v}$:
the numerical vertical wall for $-\mathbf{v}$ is the vertical line $\beta=-\frac{1}{3}$; 
the numerical semicircular walls are two sets of nested semicircles whose top points lie on the hyperbola \eqref{Kp-wall-apexes-hyperbola}.
The hyperbola \eqref{Kp-wall-apexes-hyperbola} intersects with the wall $W(\frac{1}{2})$ at the point $(\alpha,\beta)=(\frac{1}{2},\frac{1}{2})$ which is the top point of $W(\frac{1}{2})$.

Assume that $\tilde{E}$ is a tilt semistable object with Chern character $-\mathbf{v}$. 
Suppose that a wall $W^{\prime}$ is strictly above $W(\frac{1}{2})$ induced by a short exact sequence  $0 \to F \to \tilde{E} \to G \to 0$. 
Then the wall $W^{\prime}$ contains points $(\alpha, 0)$ with $\alpha > 0$. 
In particular, $0 < H \cdot \ch_{1}(F) < H \cdot \ch_{1}(\tilde{E})=H^{3}$;
this contradicts to $\ch_{1}(F)\in \ZN H$.
This means that above the wall $W(\frac{1}{2})$, $\tilde{E}$ is tilt semistable.
Since the wall
$$
 W(\CO(3H), \tilde{E})=\{ (\alpha,\beta)\in \RN_{>0}\times \RN \mid (\beta-\frac{7}{5})^{2}+\alpha^{2}=(\frac{8}{5})^{2}\},
 $$ 
so it is above the wall $W(\frac{1}{2})$, 
and thus Serre duality yields 
$$
\Ext^{3}(\tilde{E}, \CO)\cong \Hom(\CO(3H),\tilde{E}) = 0.
$$
By Hirzebruch--Riemann--Roch formula \eqref{HRR-quadric3}, 
we have $\chi(\tilde{E} ,\CO)=-3$.
Therefore, we have
$$
\hom(\tilde{E}, \CO[1])
=\hom(\tilde{E},\CO)+\hom(\tilde{E}, \CO[2])-\chi(\tilde{E} , \CO)\geq 3.
$$
Therefore, any morphism $\tilde{E} \rightarrow \CO[1]$ destabilizes $\tilde{E}$ below the wall $W(\frac{1}{2})$.

We set $r:=  \hom(\tilde{E}, \CO[1]) \geq 3$. 
Then, there is a short exact sequence of tilt semistable objects along $W(\frac{1}{2})$ given by
$
0 \to G \to \tilde{E} \to \CO^{\oplus r}[1] \to 0.
$
Then, based on Lemma \ref{actual-wall-for-torsionsh} and Proposition \ref{classify-torsion-sheaf}, 
the rest proof follows direct from that of \cite[Lemma 6.8]{BBF+}.
\end{proof}

\begin{lem}\label{converse-to-new-G-st-sh}
Let $D$ be a Weil divisor on a hyperplane section $Y\in |H|$ with Chern character $\ch(\CO_{Y}(D)) = (0,H,\frac{1}{2} H^{2},-\tfrac{1}{3} H^{3})$. 
Then $\CO_{Y}(D)$ is globally generated and $H^{0}(Y,\CO_{Y}(D)) =\CN^{3}$.
In particular, $D$ is a divisor of type $(2,0)$and $Y$ is smooth.
\end{lem}

\begin{proof}
Suppose that $D$ is a Weil divisor on a hyperplane section $Y \in |H|$ of $Q$ with Chern character 
$$
\ch(\CO_{Y}(D))=(0,H,\frac{1}{2}H^{2},-\frac{1}{3}H^{3}).
$$
By Proposition \ref{classify-torsion-sheaf}, 
the sheaf $\CO_{Y}(D)$ is tilt stable for all $\alpha > 0$ and $\beta \in \RN$.  
By a direct computation, 
the numerical wall $W(\CO_{Y}(D),\CO(-3H)[1])=W(\frac{7}{2})$. 
Hence, by Serre duality, $h^2(Y,\CO_{Y}(D))=\hom(\CO_{Y}(D),\CO(-3H)[1])=0$. 
By Hirzebruch--Riemann--Roch formula \eqref{HRR-quadric3},
we have $\chi(\CO_{Y}(D))=3$ and  
$$
h^0(Y,\CO_{Y}(D)) 
= \chi(\CO_{Y}(D))+h^{1}(Y,\CO_{Y}(D)) + h^{3}(Y,\CO_{Y}(D)) \geq \chi(\CO_{Y}(D)) = 3.
$$

Let $V \subset H^{0}(Y,\CO_{Y}(D))$ be a subspace of dimension $3$.
Then, we obtain the object $\mathcal{E}_{D,V} \in \D(Q)$.
By Proposition \ref{Construct-sh-prop}, 
the sheaf  $E_{D, V}=\CH^{-1}(\mathcal{E}_{D, V})$ is reflexive and $\mu_{H}$-slope stable. 
Suppose $\ch(\CH^{0}(\mathcal{E}_{D,V})):= (0,aH,\frac{b}{2}H^{2}, \frac{c}{12}H^{3})$ with $a\geq 0$ and $a\in \ZN$.
By the exact sequence \eqref{constrct-sheaf-exact-sequ}, 
we have
$$
\ch(E_{D,V})=(3,(a-1)H,\frac{b-1}{2}H^{2},(\frac{c}{12}+\frac{1}{3})H^{3}).
$$
Since $E_{D,V}$ is a $\mu_{H}$-slope stable sheaf, hence $a=0$ and thus $b\geq 0$.
By Proposition \ref{maximal-ch-Q}, 
we have $\frac{b-1}{2}\leq -\frac{1}{2}$ and thus $b=0$ and $c\geq 0$. 
Again Proposition \ref{maximal-ch-Q} implies that $\frac{c}{12}+\frac{1}{3}\leq \frac{1}{3}$ and thus $c=0$. Therefore, $\ch(\CH^{0}(\mathcal{E}_{D,V}))=0$ and thus $\CH^{0}(\mathcal{E}_{D,V})=0$.
This yields that $\CO_{Y}(D)$ is globally generated and $\ch(E_{D,V})=(3,-H,-\frac{1}{2}H^{2},\frac{1}{3} H^{3})$.

Since $E_{D, V}$ is $\mu_{H}$-slope stable and $\mu_{H}(E_{D,V})=-\frac{1}{3}$, 
so $\Hom(\CO,E_{D,V})=0$, and the Serre duality implies 
$H^{3}(Q,E_{D,V}) \cong \Hom(E_{D,V},\CO(-3H))=0$.  
By \eqref{H2-shEV=0}, we have $H^{2}(Q, E_{D,V})=0$. 
Using Hirzebruch--Riemann--Roch formula \eqref{HRR-quadric3},
it follows that $h^1(Q,E_{D,V})=-\chi(E_{D,V})=0$. 
Using \eqref{constrct-sheaf-exact-sequ}, 
it follows that $H^{i}(Q,\CO_{Y}(D))=0$ for $i>0$ and $H^{0}(Q,\CO_{Y}(D)) = \CN^{3}=V$.

Furthermore, 
according to Bertini's theorem, 
a general section of $\CO_{Y}(D)$ cuts out a smooth curve $C$.
By adjunction formula, 
we have
\begin{eqnarray*}
\ch(\omega_{C}) 
& = & \ch(\CO_{Y}(-2H+D)|_{D})= \ch(\CO_{Y}(-2H+D))-\ch(\CO_{Y}(-2H)) \\
& = &  (0,H,-\frac{3}{2}H^{2},\frac{2}{3}H^{3})-(0,H,-\frac{5}{2}H^{2},\frac{19}{6}H^{3}) \\
& = &  (0,0,H^{2},-\frac{5}{2}H^{3}).
\end{eqnarray*}
Hence, the smooth curve $C$ is of degree $2$.
Moreover, by Hirzebruch--Riemann--Roch formula \eqref{HRR-quadric3},
we have $\chi(\omega_{C})=-2$.
This means that $D$ is divisor of type $(2,0)$ (i.e., it is either a double line or a disjoint union of two lines) and hence $Y$ is smooth.
\end{proof}

As a by-product, we have the following

\begin{cor}\label{VB-in-Ku}
For any divisor $D$ of type $(2,0)$ in a smooth hyperplane section $Y\subset Q$,
the associated vector bundle $E_{D}\in \Ku(Q)$.
\end{cor}

\begin{proof}
From the proof of Lemma \ref{converse-to-new-G-st-sh}, 
it follows that $\Hom(\CO,E_{D}[i])=0$ for any $i\in \ZN$. 
The rest is to show $\Hom(\CO(H),E_{D}[i])=0$ for any $i\in \ZN$. 
Since $E_{D}$ is $\mu_{H}$-slope stable and $\mu_{H}(E_{D})=-\frac{1}{3}<0=\mu_{H}(\CO)$, 
so $\Hom(\CO,E_{D})=0$.
Similarly, since $\mu_{H}(E_{D})=-\frac{1}{3}>-2=\mu_{H}(\CO(-2H))$, by Serre duality, 
we have $\Hom(\CO(H),E_{D}[3]) \cong \Hom(E_{D}, \CO(-2H))=0$.
Note that
$$
0=\chi(\CO(H),E_{D})= \ext^{1}(\CO(H),E_{D})-\ext^{2}(\CO(H),E_{D}).
$$
It is sufficient to show that $\Ext^{2}(\CO(H),E_{D})=0$.
Again, since $E_{D}$ is $\mu_{H}$-slope stable and $\mu_{H}(E_{D})=-\frac{1}{3}$,
by Proposition \ref{limit-tilt-stab-Gie-stab-prop}, 
$E_{D} \in \Coh^{\beta}(X)$ is $\sigma_{\alpha, \beta}$-stable for $\alpha \gg 0$ and $\beta < -\frac{1}{3}$. 
Hence, by Lemma \ref{no-wall-1}, $E_{D}$ is $\sigma_{\alpha,-1}$-stable for $\alpha>0$. 
By \cite[Corollary 3.11]{BMS16}, $\CO(-2H)[1]$ is also $\sigma_{\alpha,-1}$-stable for $\alpha>0$. 
Note that 
$$
\lim_{\alpha\to 0^{+}} \mu_{\alpha, -1}(\CO(-2H)[1])=\frac{\alpha^{2}-1}{2}=-\frac{1}{2} < 0=\lim_{\alpha\to 0^{+}} -\frac{3\alpha^{2}}{4}=\mu_{\alpha, -1}(E_{D}).
$$
By  Serre duality, it follows that $\Ext^{2}(\CO(H),E_{D})\cong \Hom(E_{D},\CO(-2H)[1])=0$.
\end{proof}

The following proposition implies that the vector bundle $E_{D}$ is $\sigma$-stable for all $\sigma\in \mathcal{K}$.

\begin{prop}\label{E_D-BSC-stable}
Let $D$ be a divisor of type $(2,0)$on a smooth hyperplane section $Y$ of $Q$.
Then the sheaf $E_{D}\in \Ku(Q)$ and it is $\sigma(\alpha,\beta)$-stable for any $(\alpha,\beta)\in \widetilde{V}_{L}$, 
and $E_{D}[1]$ is $\sigma(\alpha,\beta)$-stable for any $(\alpha,\beta)\in \widetilde{V}_{R}$.
In particular, $\mathcal{P}_{x}$ and $E_{D}$ are $\sigma$-stable objects in $\Ku(Q)$ for all stability condition $\sigma\in \mathcal{K}$.
\end{prop}

\begin{proof}\label{E_D-sh-in-Ku}
By Corollary \ref{VB-in-Ku},
the vector bundle $E_{D}\in \Ku(Q)$.
By Theorem \ref{claasifying-sh-ch-v},
$E_{D}$ is a $\mu_{H}$-slope stable vector bundle.
Then the proposition follows the same proof as that of Proposition \ref{start-point-stable}.
\end{proof}


\section{Proof of Theorem \ref{main-thm-iso}}

This section is devoted to the proof of Theorem \ref{main-thm-iso}.
The proof is divided into the following three propositions (Propositions \ref{main-thm-iso-prop1}, \ref{main-thm-iso-prop2} and \ref{main-thm-iso-prop3}).

The below proposition obtains the smoothness of the moduli space.

\begin{prop}\label{main-thm-iso-prop1}
The moduli space $\overline{M}_{Q}(\mathbf{v})$ is smooth and of dimension $4$.
\end{prop}

\begin{proof}
For a $H$-Gieseker semistable sheaf $E$ with $\ch(E)=\mathbf{v}$,
since $\ch_{\leq 1}(E)=(3,-H)$ is primitive, 
the tangent space of $\overline{M}_{Q}(\mathbf{v})$  at $E$ is given by $\Ext^{1}(E,E)$ and the obstruction lies in $\Ext^{2}(E,E)$; 
if $\Ext^{2}(E,E)=0$, 
then $\overline{M}_{Q}(\mathbf{v})$ is smooth at the point $[E]$ (cf. \cite[Corollary 4.5.2]{HL10}).
Therefore, it is sufficient to show that for every $H$-Gieseker semistable sheaf $E$ with Chern character $\ch(E)=\mathbf{v}$,
$$
\Ext^{i}(E,E) =
\begin{cases}
\CN & \textrm{if } i = 0 \\
\CN^{4} & \textrm{if } i =1 \\
0 &\textrm{otherwise.}
\end{cases}
$$
Since $\ch_{\leq 1}(E)=(3,-H)$ is primitive, 
hence $E$ is $\mu_{H}$-slope stable and so is it twists $E(-3H)$.
Therefore, $\Hom(E,E)=\CN$.
Since $\mu_{H}(E)=-\frac{1}{3}>-\frac{10}{3}=\mu_{H}(E(-3H))$, 
the Serre duality implies $\Ext^{3}(E,E)\cong \Hom(E, E(-3H))^{\vee}=0$.
By Hirzebruch--Riemann--Roch formula \eqref{HRR-quadric3}, we have $\chi(E,E)=-3$ and thus
$$
\ext^{1}(E,E)=4+\ext^{2}(E,E).
$$
Next, we show $\ext^{2}(E,E)=0$.
By Serre duality, $\Ext^{2}(E,E)\cong \Hom(E, E(-3H)[1])$.
In the following, we need to prove $\Hom(E, E(-3H)[1])=0$.
Since $E$ is $\sigma_{\alpha,-1}$-stable for $\alpha \gg 0$,
by Lemma \ref{no-wall-1},
$E$ is $\sigma_{\alpha, -1}$-stable for any $\alpha > 0$. 

By Proposition \ref{maximal-ch-Q}, 
$E$ is a $\mu_{H}$-slope stable reflexive sheaf and thus it twists $E(-3H)$ is also $\mu_{H}$-slope stable reflexive.
Note that $\mu_{H}(E(-3H))=-\frac{10}{3}<-1$.
By \cite[Lemma 2.7 (c)]{BMS16}, 
we have $E(-3H)[1]\in \Coh^{-1}(Q)$.
Moreover, by Lemma \ref{classify-tor-lem1}, 
$E[1]$ is $\sigma_{\alpha,2}$-stable for any $\alpha > 0$ 
and thus $E(-3H)[1]$ is $\sigma_{\alpha,-1}$-stable for $\alpha > 0$ .  
Because of the limit tilt slopes
$$
\lim_{\alpha \to 0^{+}} \mu_{\alpha,-1}(E)
=\lim_{\alpha \to 0^{+}} -\frac{3\alpha^{2}}{4}> \lim_{\alpha \to 0^{+}}\frac{3\alpha^{2}-15}{14}=\lim_{\alpha \to 0^{+}} \mu_{\alpha,-1}(E(-3H)[1])
$$ 
and both $E$ and $E(-3H)[1]$ are $\sigma_{\alpha, -1}$-stable for $\alpha > 0$,
it follows that 
$$
\Ext^{2}(E,E)\cong \Hom(E, E(-3H)[1])=0.
$$
This finishes the proof.
\end{proof}

Next we shall gain the irreducibility of the moduli space.

\begin{prop}\label{main-thm-iso-prop2}
The moduli space $\overline{M}_{Q}(\mathbf{v})$ is irreducible.
\end{prop}

\begin{proof}
Let $\mathrm{Hilb}_{Q}^{2m+2}$ be the Hilbert scheme of curves on $Q$ with Hilbert polynomial $2m+2$ and $\mathrm{Hilb}_{Q,+}^{2m+2}\subset \mathrm{Hilb}_{Q}^{2m+2}$ be the subspace parametrizing the locally Cohen-Macaulay curves in $\mathrm{Hilb}_{Q}^{2m+2}$.
Then, by \cite[Proposition 5.31]{BH18}, 
$\mathrm{Hilb}_{Q,+}^{2m+2}$ is smooth and irreducible of dimension $6$. 
In fact, for a point $[D]\in \mathrm{Hilb}_{Q,+}^{2m+2}$, $D$ is either the disjoint union of two lines or a double line.
Recall that $M_{Q}(\mathbf{v})\subset \overline{M}_{Q}(\mathbf{v})$ is the open locus of $H$-Gieseker semistable vector bundles on $Q$ with Chern character  $\mathbf{v}$.
By Theorem \ref{claasifying-sh-ch-v}, 
there is a surjective morphism $\mathrm{Hilb}_{Q,+}^{2m+2} \rightarrow M_{Q}(\mathbf{v})$ which sends $[D]\in \mathrm{Hilb}_{Q,+}^{2m+2}$ to the vector bundle $E_{D}$.
Therefore, we obtain that $M_{Q}(\mathbf{v})$ is irreducible and thus $\overline{M}_{Q}(\mathbf{v})$ is also irreducible.
\end{proof}

\begin{rem}
Let $V_{3}$ be a smooth cubic threefold and 
$\overline{M}_{V_{3}}(v)$ the moduli space of Gieseker-semistable sheaves on $V_{3}$ with Chern character
$v:=(3, -H, -\frac{1}{2}H^{2}, \frac{1}{6}H^{3})$.
Then $\overline{M}_{V_{3}}(v)$ is smooth and irreducible of dimension $4$; moreover, $\overline{M}_{V_{3}}(v)$ is isomorphic to the blow up of the theta divisor $\Theta$ at its unique singular point (see \cite[Theorem 1.1]{BBF+}). 
Using the same idea for proving \cite[Lemma 7.5]{BBF+},
there also exists a contraction from $\overline{M}_{Q}$ to an algebraic space $Y$ of finite type over $\CN$ such that it is an isomorphism away from $Q$ and contracts $Q$ to a point $0 \in Y$.
However, it is not clear that if this algebraic space $Y$ can be described explicitly.
\end{rem}

\begin{lem}\label{Q-embed-Giseker-mod}
The natural morphism
\begin{eqnarray*}
\iota: Q &\longrightarrow&\overline{M}_{Q}(\mathbf{v})  \\
         x &\longmapsto& \mathcal{P}_{x}
\end{eqnarray*}
is an embedding with the normal bundle $\CO(-2H)$.
\end{lem}

\begin{proof}
The smooth quadric threefold $Q$ can be viewed as the moduli space of twisted ideal sheaves $I_{x}(H)$ for all $x\in Q$.
By Proposition \ref{main-thm-iso-prop1}, $\overline{M}_{Q}(\mathbf{v})$ is smooth.
Since both $Q$ and $\overline{M}_{Q}(\mathbf{v})$ are smooth,
then the proof  is the same as that of \cite[Lemma 7.3]{BBF+}.
\end{proof}


Finally, we prove that the Bridgeland moduli space $M_{\sigma}(2\lambda_{2}-\lambda_{1})$ 
is isomorphic the Gieseker moduli space $\overline{M}_{Q}(\mathbf{v})$.

\begin{prop}\label{main-thm-iso-prop3}
For any stability condition $\sigma\in \mathcal{K}$, the moduli space $M_{\sigma}(2\lambda_{2}-\lambda_{1})$ is isomorphic to the moduli space $\overline{M}_{Q}(\mathbf{v})$.
In particular, the moduli space $M_{\sigma}(2\lambda_{2}-\lambda_{1})$ is a smooth complex projective variety of dimension $4$.
\end{prop}

To this end,
we first give a characterization of $\sigma$-stable objects in the Bridgeland moduli space $M_{\sigma}(2\lambda_{2}-\lambda_{1})$.

\begin{lem}\label{charact-stabob-inKu}
Let $\sigma \in \mathcal{K}$ be a stability condition.
If $E$ is a $\sigma$-stable object in $\Ku(Q)$ with class $2\lambda_{2}-\lambda_{1}$,
then up to a shift, $E$ is either the sheaf $\mathcal{P}_{x}$ or the vector bundle $E_{D}$  in Theorem \ref{claasifying-sh-ch-v}.
\end{lem}

\begin{proof}
By hypothesis, the Chern character of $E$ is $\ch(E)=\pm(3,-H,-\frac{1}{2}H^{2},\frac{1}{3}H^{3})$.
By Lemma \ref{GL-one-orbitV+},
we may assume that $E\in \Ku(Q)$ is $\sigma(\alpha,\beta)$-stable object with class $2\lambda_{2}-\lambda_{1}$ for $(\alpha,\beta)\in V_{L}:=V\cap \widetilde{V}_{L}$.
Hence, we may assume that $G:= E[1]\in \A(\alpha,\beta)$ is $\sigma(\alpha,\beta)$-stable with the slope
$$
\mu_{\alpha,\beta}^{0}(G)=\frac{-(3\beta+1)}{\frac{1}{2}-\beta-\frac{3}{2}\beta^{2}+\frac{3}{2}\alpha^{2}}.
$$
It is important to note that for $0<\alpha<\frac{1}{2}$, $(\alpha,\alpha-1)\in V_{L}$ and the limit of $\mu_{\alpha,\alpha-1}^{0}$-slope
$$
\lim_{\alpha \to 0^{+}} \mu_{\alpha,\alpha-1}^{0}(G)
= \lim_{\alpha \to 0^{+} } (\frac{1}{\alpha}-\frac{3}{2})=+\infty.
$$

\begin{claim}\label{claim-for-charact-stabob-inKu}
$G$ is $\sigma_{\alpha,\beta}^{0}$-semistable for some $(\alpha,\beta)\in V_{L}$.
\end{claim}

Temporarily admitting this claim,
using \cite[Proposition 4.1]{FP23} which is also valid for smooth quadric threefold,
there is either (i) $G$ is $\sigma_{\alpha,\beta}$-semistable; 
or (ii)  there is a distinguished triangle
$$
\xymatrix@C=0.5cm{
F[1] \ar[r]^{} & G \ar[r]^{} & T }
$$
where $F\in \mathcal{F}^{0}_{\alpha,\beta}$ and $T\in \mathcal{T}^{0}_{\alpha,\beta}$ is supported on points. 

In the case (i), we may assume $G:=E$.
By hypothesis, $G$ is $\sigma_{\alpha,\beta}$-semistable.
By Lemma \ref{no-wall-1}, 
there is no wall for objects with $\ch_{\leq 2}(G)=(3,-H,-\frac{1}{2}H^{2})$ to the left of the vertical wall $\beta=-\frac{1}{3}$.
Hence, $G$ is $\sigma_{\alpha,\beta}$-semistable for $\alpha \gg 0$.
By Corollary \ref{limit-tilt-stab-Gie-stab},
$G$ is a $\mu_{H}$-slope stable torsion-free sheaf.
By Theorem \ref{claasifying-sh-ch-v}, it follows that the assertion holds.

In the case (ii), we may assume that $G:=E[1]$.
Since $G$ is $\sigma_{\alpha,\beta}^{0}$-semistable,
it follows that $F$ must be $\sigma_{\alpha,\beta}$-semistable.
Therefore, we may assume that the Chern character of $F$ is
$$
\ch(F)=(3,-H,-\frac{1}{2}H^{2},(\frac{1}{3}+\frac{t}{2})H^{3}),
$$
where $t\geq 0$. 
By Lemma \ref{no-wall-1}, 
there is no wall for $F$ with $\ch_{\leq 2}(F)=(3,-H,-\frac{1}{2}H^{2})$ to the left of the vertical wall $\beta=-\frac{1}{3}$ and thus $F$ is $\sigma_{\alpha,\beta}$-semistable for $\alpha\gg0$.
Since $\mu_{H}(F)=-\frac{1}{3}>\beta$, 
by Corollary \ref{limit-tilt-stab-Gie-stab}, $F$ is a $\mu_{H}$-slope stable torsion-free sheaf.
By Proposition \ref{maximal-ch-Q},  
$\ch_{3}(F)=(\frac{1}{3}+\frac{t}{2})H^{3}\leq \frac{1}{3}H^{3}$, and since $t\geq 0$, so $t=0$. 
This means that $T=0$ and $G=F[1]$. 
As a result, we are reduced to the case (i).
This concludes the proof of Lemma \ref{charact-stabob-inKu}.
\end{proof}

Next we present the proof of Claim \ref{claim-for-charact-stabob-inKu}.

\begin{proof}[Proof of Claim \ref{claim-for-charact-stabob-inKu}]
We conclude by contradiction using the same idea as \cite[Proposition 4.1]{FLZ23}).
Suppose that $G$ is not $\sigma_{\alpha,\beta}^{0}$-semistable for all $(\alpha,\beta)\in V_{L}$.
Since $G$ is $\sigma(\alpha,\beta)$-stable, 
by \cite[Proposition 2.2.2]{BMT14}, 
there is an open ball $U$ in $\RN^{2}$ containing the point $(0,-1)$ such that for any $(\alpha,\beta)\in U\cap V_{L}$, 
$G\in \A(\alpha,\beta)$ and $G$ has constant HN filtration with respect to $\sigma^{0}_{\alpha,\beta}$.
Suppose that $B$ is the destabilizing quotient object of $G$ with minimum slope and 
$A\to G\to B$ 
is the destabilizing sequence of $E$ with respect to $\sigma_{\alpha,\beta}^{0}$ for $(\alpha,\beta)\in U\cap V_{L}$,
where $A,B\in \Coh_{\alpha,\beta}^{0}(Q)$.

We set $\ch_{\leq 2}(B):=(a,bH,\frac{c}{2}H^2)$ with $a,b,c\in \ZN$.
Note that $\mu_{\alpha,\beta}^{0}(G)>\mu_{\alpha,\beta}^{0}(B)$ and 
$\Im(Z_{\alpha,\beta}^{0}(G))\geq \Im(Z_{\alpha,\beta}^{0}(B))>0$  for any $(\alpha,\beta)\in U\cap V_{L}$. 
Since $\displaystyle \lim_{\alpha \to 0^{+}} Z_{\alpha,\alpha-1}^{0}(G)=0$,
it follows that $\displaystyle \lim_{\alpha \to 0^{+}} Z_{\alpha,\alpha-1}^{0}(B)=0$ and thus $c=-a-2b$. 
Therefore, we have
\begin{equation}\label{claim-im-ineq}
0<(\beta^{2}-\alpha^{2}-1)\frac{a}{2}-(\beta+1)b\leq \frac{1}{2}-\beta-\frac{3}{2}\beta^{2}+\frac{3}{2}\alpha^{2}
\end{equation} 
and
\begin{equation}\label{claim-mu-slope-ineq}
\frac{(\beta a-b)}{(\beta^{2}-\alpha^{2}-1)\frac{a}{2}-(\beta+1)b}
<
\frac{-(3\beta+1)}{\frac{1}{2}-\beta-\frac{3}{2}\beta^{2}+\frac{3}{2}\alpha^{2}}
\end{equation}
for all $(\alpha,\beta)\in U\cap V_{L}$.
As a result,
we have
\begin{equation}\label{claim-ineq-comb}
 \beta a-b < -(3\beta+1).
\end{equation}
By \cite[Remark 5.12]{BLMS23}, 
one has $\mu_{\alpha,\beta}^0(B) \geq -2$ for $(\alpha,\beta)\in V_{L}$ .
Considering the limit cases of \eqref{claim-mu-slope-ineq} and \eqref{claim-ineq-comb} when $(\alpha,\beta) \to (0,-1)$, we gain $a+b=0$, $-1$, or $-2$. 

(1) If $a+b=0$, then it contradicts to \eqref{claim-mu-slope-ineq}.  

(2) If $a+b=-1$, then $\ch_{\leq 2}(B)=(a,(-a-1)H,\frac{a+2}{2}H^{2})$.
Since $B$ is $\sigma_{\alpha,\beta}^{0}$-semistable, 
by \cite[Proposition 4.1]{FP23}, 
$\ch_{\leq 2}(B)$ is a possible truncated Chern character of a $\sigma_{\alpha,\beta}$-semistable object $B'[1]$ for $B' \in \Coh^{\beta}(Q)$.
Note that for the quadric threefold $Q$, the open region $R_{3/2rd}=R_{1/4}$ in \cite[Proposition 3.2]{Li19}.
It follows that $|a|\leq 2$.
By direct computations, it follows that $a=\pm 1$ and $a=2$ contradict to \eqref{claim-mu-slope-ineq}.
If $a=-2$, then $\ch_{\leq 2}(B)=(-2,H,0)$ and $\ch_{\leq 2}(A)=(-1,0,\frac{1}{2}H^{2})$.

\begin{claim}\label{ch-B-part}
$\ch(B)=(-2,H,0,-\frac{1}{12}H^{3})$ and thus $\ch(A)=(-1,0,\frac{1}{2}H^{2},-\frac{1}{4}H^{3})$.
\end{claim}

\begin{proof}
By Hirzebruch--Riemann--Roch formula \eqref{HRR-quadric3}
it is sufficient to show that $\Hom(\CO, B[i])=0$ for all $i\in \ZN$.  
Since $\CO, \CO(-3H)[2]\in \Coh^0_{\alpha,\beta}(X)$, 
by Serre duality, 
we get $\Hom(\CO,B[i])\cong \Hom(B,\CO(-3H)[3-i])=0$ for $i\neq 0,1$. 
Note that $\displaystyle \lim_{(\alpha,\beta) \to (0,-1)}\mu^{0}_{\alpha,\beta}(B)=+\infty$.
By shrinking the open ball $U$, we may assume  
$
(\mu^{0}_{\alpha,\beta})^{-}(A) > \mu^{0}_{\alpha,\beta}(B) > \mu^{0}_{\alpha,\beta}(\CO(-3H)[2]).
$
Based on the $\sigma_{\alpha,\beta}^{0}$-semistability,
Serre duality yields that $\Hom(\CO, B[1])\cong \Hom(B,\CO(-3H)[2])=0$ and $\Hom(\CO, A[1])\cong\Hom(A,\CO(-3H)[2]) = 0$.
Moreover, it follows from $E\in \Ku(Q)$ that $\Hom(\CO, B)\cong \Hom(\CO, A[1])=0$. 
This proves the claim.
\end{proof}

Since $B$ is $\sigma_{\alpha,\beta}^{0}$-semistable, 
by Proposition \ref{semistab-char-spinorbd}, we obtain $B\cong S[1]$.
Since $A$ is $\sigma_{\alpha,\beta}^{0}$-semistable, 
as shown in Theorem \ref{Hilb-line-Q}, 
$A\cong I_{\ell}[1]$ for some line $\ell\subset Q$.
Therefore, there is a short exact sequence $0\to I_{\ell} \to G[-1]\to S \to 0$ and thus $G[-1]\cong I_{\ell}\oplus S$, as $\Ext^{1}(S,I_{\ell})=0$.
This contradicts the stability.

(3) If $a+b=-2$, then $\ch_{\leq 2}(B)=(a,(-a-2)H,\frac{a+4}{2}H^{2})$.
By \cite[Proposition 3.2]{Li19} again, one has $|a|\leq 3$.
Since $\beta>-1$, by \eqref{claim-ineq-comb}, we have $a<-3$, a contradiction.
This finishes the proof of Claim \ref{claim-for-charact-stabob-inKu}.
\end{proof}

Now we are in the position to finish the proof of Proposition \ref{main-thm-iso-prop3}.

\begin{proof}[Proof of Proposition \ref{main-thm-iso-prop3}]
Let us begin with some facts on the two moduli spaces.
\begin{itemize}
\item 
By \cite[Theorem 21.24 (2)]{BLMNPS21},
the moduli stack $\mathcal{M}_{\sigma}(2\lambda_{2}-\lambda_{1})$ is an algebraic stack of finite type over $\CN$ for any stability condition $\sigma \in \mathcal{K}$; see \cite[Definition 21.11]{BLMNPS21} for the definition of the moduli stack.
Since the class $2\lambda_{2}-\lambda_{1}$ is primitive, 
this yields that stability and semistability coincide for objects of the class $2\lambda_{2}-\lambda_{1}$.
By \cite[Theorem 21.24 (2)]{BLMNPS21}, 
$\mathcal{M}_{\sigma}(2\lambda_{2}-\lambda_{1})$ admits a coarse moduli space $M_{\sigma}(2\lambda_{2}-\lambda_{1})$ which is an algebraic space proper over $\CN$. 
Moreover, by \cite[Theorem 21.24 (3)]{BLMNPS21}, 
$M_{\sigma}(2\lambda_{2}-\lambda_{1})$ is also a good moduli space of $\mathcal{M}_{\sigma}(2\lambda_{2}-\lambda_{1})$.

 \item 
Let $\overline{\mathcal{M}}_{Q}(\mathbf{v})$ be the the moduli stack of $H$-Gieseker semistable sheaves with the Chern character $\mathbf{v}$.
Since $\mathbf{v}$ is primitive,
then $\overline{\mathcal{M}}_{Q}(\mathbf{v})$ admits a good moduli space $\overline{M}_{Q}(\mathbf{v})$
which is also a coarse moduli space;  see \cite[Example 8.7]{Alp13}.
\end{itemize}

By Proposition \ref{start-point-stable} and Lemma \ref{E_D-BSC-stable},
for a closed point $x\in Q$ and a divisor $D$ of type $(2,0)$on a smooth hyperplane section of $Q$,
the sheaf $\mathcal{P}_{x}$ and the vector bundle $E_{D}$ are $\sigma(\alpha,\beta)$-stable for any $(\alpha,\beta)\in \widetilde{V}_{L}$;
$\mathcal{P}_{x}[1]$ and $E_{D}[1]$ are $\sigma(\alpha,\beta)$-stable for any $(\alpha,\beta)\in \widetilde{V}_{R}$.
Therefore, $\mathcal{P}_{x}$ and  $E_{D}$ are $\sigma$-stable for any  $\sigma \in \mathcal{K}$.
Conversely, by Lemma \ref{charact-stabob-inKu}, 
for any $\sigma$-stable object in $\Ku(Q)$ with class $2\lambda_{2}-\lambda_{1}$,  up to a shift, it is either the sheaf $\mathcal{P}_{x}$ or the vector bundle $E_{D}$. 
Therefore, we may identify these two moduli stacks $\mathcal{M}_{\sigma}(2\lambda_{2}-\lambda_{1})$ and $\overline{\mathcal{M}}_{Q}(\mathbf{v})$. 
As a consequence, by the uniqueness of coarse moduli spaces,
the moduli space $M_{\sigma}(2\lambda_{2}-\lambda_{1})$ is isomorphic to the moduli space $\overline{M}_{Q}(\mathbf{v})$.
This completes the proof of Proposition  \ref{main-thm-iso-prop3} and thus the proof of Theorem \ref{main-thm-iso}.
\end{proof}

\begin{rem}
As mentioned previously, $\Ku(Q)$ is equivalent to the derived category of the Kronecker quiver $K(4)$.
Therefore, the moduli space $M_{\sigma}(2\lambda_{2}-\lambda_{1})$ may be realized as a King's moduli space  of finite-dimensional $K(4)$-representations as in \cite{Kin94}.
\end{rem}

As a direct consequence of Theorem \ref{main-thm-iso} and Lemma \ref{Q-embed-Giseker-mod},
we have:

\begin{cor}\label{Q-embed-Bridgeland-mod}
There is well-defined embedding
\begin{eqnarray}\label{embed-morph-BMod}
\Phi_{\sigma}: Q &\longrightarrow& M_{\sigma}(2\lambda_{2}-\lambda_{1})  \\
         x &\longmapsto& \mathcal{P}_{x} \nonumber
\end{eqnarray}
for every stability condition $\sigma\in \mathcal{K}$.
\end{cor}


\section{Application}\label{app-sect}

In this section,
we will show that the smooth quadric threefold $Q$ can be reinterpreted as a Brill--Noether locus in a Bridgeland moduli space on the Kuznetsov component $\Ku(Q)$.
For this purpose, we consider the Euler exact sequence in $\PB^{4}$.
By restricting the Euler exact sequence to $Q$,
there is a short exact sequence
\begin{equation}\label{U-bundle}
\xymatrix@C=0.5cm{
0 \ar[r]^{} & \mathfrak{U}
\ar[r]^{} & H^{0}(Q,\CO(H))\otimes \CO
\ar[r]^{} & \CO(H) \ar[r]^{} & 0,} 
\end{equation}
where $H^{0}(Q,\CO(H))=\CN^{5}$, and $\mathfrak{U}:=  \Omega_{\PB^{4}}|_{Q}\otimes \CO(H)$ is a vector bundle of rank $4$.
Hence, the Chern character of $\mathfrak{U}$ is 
$$
\ch(\mathfrak{U})=(4,-H,-\frac{1}{2}H^{2},-\frac{1}{6}H^{3}).
$$
Since $H^{0}(Q,\mathfrak{U})=0$, similar to Proposition \ref{Kp-slope-stable},
it is not difficult to see that the vector bundle $\mathfrak{U}$ is $\mu_{H}$-slope stable.
Moreover, applying Serre functor to $\CO(H)$, we have $\D(Q)=\langle \Ku(Q),\CO,  \CO(H)\rangle=\langle \CO(-2H), \Ku(Q),\CO \rangle$.
By the definition of left mutation, it follows
$$
\mathfrak{U} \cong \mathrm{L}_{\CO}(\CO(H))[-1] \in \langle\CO(-2H), \Ku(Q)\rangle.
$$ 
In fact, in \cite{Yan24}, we will show that the vector bundle $\mathfrak{U}$ is the unique $\mu_{H}$-slope stable sheaf with Chern character $(4,-H,-\frac{1}{2}H^{2},-\frac{1}{6}H^{3})$.

Next, we will project the vector bundle $\mathfrak{U}$ into the Kuznetsov component $\Ku(Q)$.
We denote by $\mathfrak{U}_{Q}:=  \mathrm{R}_{\CO(-2H)} \mathfrak{U}\in \Ku(Q)$ the right mutation of $\mathfrak{U}$ through $\CO(-2H)$.
Using the exact sequence \eqref{U-bundle}, 
a direct computation shows  $\RHom(\mathfrak{U},\CO(-2H))=\CN[-2]$.
Hence, there is a distinguished triangle
\begin{equation}\label{U-bundle-project}
\xymatrix@C=0.5cm{ 
\CO(-2H)[1] \ar[r]^{} & \mathfrak{U}_{Q} \ar[r]^{} & \mathfrak{U} \ar[r]^{} &  \CO(-2H)[2]}
\end{equation} 
in $\D(Q)$.

Likewise to \cite[Definition 6.2]{JLZ22} and \cite[Theorem 6.2]{FLZ23}, 
we give the following:

\begin{defn}
For any stability condition $\sigma\in \mathcal{K}$,
the {\it Brill--Noether locus} in the moduli space $M_{\sigma}(2\lambda_{2}-\lambda_{1})$ (with respect to $\mathfrak{U}_{Q}\in \Ku(Q)$) is given by
$$
\mathcal{BN}(Q):=
\{E\in M_{\sigma}(2\lambda_{2}-\lambda_{1}) \mid \dim \Ext^{1}(E, \mathfrak{U}_{Q})=4 \}.
$$
\end{defn}

Now we can state the main result in this section.

\begin{thm}\label{Q-BN-locus}
For every stability condition $\sigma\in \mathcal{K}$, 
the morphism $\Phi_{\sigma}$ in \eqref{embed-morph-BMod} induces an isomorphism from the quadric threefold $Q$ to the Brill--Noether locus $\mathcal{BN}(Q)$.
\end{thm}

\begin{proof}
By Theorem \ref{main-thm-iso}, 
it suffices to show that $\mathcal{P}_{x}\in \mathcal{BN}(Q)$ and $E_{D}\notin \mathcal{BN}(Q)$, 
for a closed point $x\in Q$ and a divisor $D$ of type $(2,0)$in a smooth hyperplane section $Y\subset Q$.

We first compute $\Ext^{i}(\mathcal{P}_{x},\mathfrak{U}_{Q})$.
Since $\mathcal{P}_{x}\in \Ku(Q)$, by the exact triangle \eqref{U-bundle-project}, 
we have 
$$
\RHom(\mathcal{P}_{x},\mathfrak{U}_{Q})\cong \RHom(\mathcal{P}_{x},\mathfrak{U}).
$$
By the exact sequence \eqref{structure-exact-cor-Q}, 
we gain $\RHom(\mathcal{P}_{x},\mathfrak{U})\cong \RHom(I_{x}(H),\mathfrak{U})[1]$.
By Grothendieck duality, we get $\RHom(\CO_{x},\mathfrak{U})=\CN^{4}[-3]$.
Using the exact sequence \eqref{U-bundle}, 
we have $\RHom(\CO(H),\mathfrak{U})=\CN[-1]$.
Applying $\Hom(-,\mathfrak{U})$ to the exact sequence \eqref{structure-exact-Q}, 
we have a distinguished triangle
$$
\xymatrix@C=0.5cm{ 
\RHom(\CO_{x},\mathfrak{U}) \ar[r]^{} & \RHom(\CO(H),\mathfrak{U})  \ar[r]^{} & \RHom(I_{x}(H),\mathfrak{U}).}
$$
it follows that $\RHom(I_{x}(H),\mathfrak{U})=\CN[-1]\oplus \CN^{4}[-2]$. 
Therefore, we get $\RHom(\mathcal{P}_{x},\mathfrak{U})=\CN[0]\oplus \CN^{4}[-1]$ and $\mathcal{P}_{x}\in \mathcal{BN}(Q)$.

Next, we will compute $\Ext^{1}(E_{D},\mathfrak{U}_{Q})$.
Since $E_{D}\in \Ku(Q)$, by the exact triangle \eqref{U-bundle-project}, 
we have
$$
\RHom(E_{D},\mathfrak{U}_{Q})\cong \RHom(E_{D},\mathfrak{U}).
$$
For the divisor $D$, by Proposition \ref{new-G-st-sh}, we have 
\begin{equation}\label{VectE-9-sequ}
\xymatrix@C=0.5cm{
0 \ar[r]^{} & E_{D}
\ar[r]^{} &  \CO^{\oplus 3}
 \ar[r]^{} & \CO_{Y}(D)
 \ar[r]^{} & 0,}
\end{equation}
where $Y\subset Q$ is a smooth hyperplane section and $D\subset Y$ is a divisor of $(2,0)$ type.
Since $\mathfrak{U}$ is $\mu_{H}$-slope stable, 
so $\Hom(\CO^{\oplus 3}, \mathfrak{U})=0$,  and Serre duality implies
$$
\Ext^{3}(\CO^{\oplus 3}, \mathfrak{U})\cong \Hom(\mathfrak{U},\CO(-3H)^{\oplus 3})=0.
$$
Applying $\Hom(\CO^{\oplus 3},-)$ to \eqref{U-bundle}, 
we obtain $\Ext^{1}(\CO^{\oplus 3}, \mathfrak{U})=0$ 
and $\Ext^{2}(\CO^{\oplus 3}, \mathfrak{U})=0$,
as $\{\CO,\CO(H)\}$ is a strongly exceptional pair.
Moreover, since $\Hom(\CO^{\oplus 3},\mathfrak{U})=0$ and $\CO_{Y}(D)$ is a torsion sheaf on $Q$,
by applying $\Hom(-,\mathfrak{U})$ to the exact sequence \eqref{VectE-9-sequ},
we get $\Hom(E_{D},\mathfrak{U})=0$, $\Ext^{1}(E_{D},\mathfrak{U})=0$,
and 
$$
\Ext^{2}(E_{D},\mathfrak{U})\cong \Ext^{1}(\CO_{Y}(D),\mathfrak{U})
\;
\textrm{and}
\;
\Ext^{3}(E_{D},\mathfrak{U})\cong \Ext^{2}(\CO_{Y}(D),\mathfrak{U}).
$$

Since $Y\subset Q$ is smooth and $D\subset Y$ is a divisor of $(2,0)$ type,
by Serre duality, we get $H^{3-i}(Y,\CO_{Y}(D-3H))\cong H^{i-1}(Y,\CO_{Y}(H-D))=0$ for $i=1,2$.
As a result, for $i=1,2$, by Serre duality on $Q$, we have
$$
\Ext^{i}(\CO_{Y}(D),\CO) 
\cong 
\Ext^{3-i}(\CO,\CO_{Y}(D)\otimes \CO(-3H))
\cong 
H^{3-i}(Y,\CO_{Y}(D-3H))=0.
$$
Similarly, we have 
$\Ext^{1}(\CO_{Y}(D),\CO(H)) \cong H^{0}(Y,\CO_{Y}(2H-D))=\CN^{3}$.
Next, by applying $\Hom(\CO_{Y}(D),-)$ to the exact sequence \eqref{U-bundle}, 
we derive $\Ext^{1}(\CO_{Y}(D),\mathfrak{U})=0$
and
$$
 \Ext^{2}(\CO_{Y}(D),\mathfrak{U})\cong \Ext^{1}(\CO_{Y}(D),\CO(H))=\CN^{3}.
$$
Therefore, we have $\RHom(E_{D},\mathfrak{U})=\CN^{3}[-3]$ and thus $E_{D}\notin \mathcal{BN}(Q)$.
This completes the proof of Theorem \ref{Q-BN-locus}.
\end{proof}


\appendix

\section{Uniqueness of stable spinor bundle}\label{Append}

In this appendix, we will discuss the uniqueness of Bridgeland stable objects in $\Ku(Q)$ with the numerical class of the spinor bundle $S$.
By \cite[Theorem 2.1]{Ott88}, we know that the spinor bundle $S$ is $\mu_{H}$-slope stable;
moreover, any $\mu_{H}$-stable vector bundle of rank $2$ with Chern classes $c_{1}=-1$ and $c_{2}=1$ is isomorphic to the spinor bundle $S$; see \cite{AS89,OS94} or \cite[Proposition 40]{CJ24}.
The below lemma shows that $S$ is also a Bridgeland stable object in $\Ku(Q)$.

\begin{lem}\label{spinor-bd-stable}
The spinor bundle $S$ is $\sigma$-stable for all stability conditions $\sigma\in \mathcal{K}$.
\end{lem}

\begin{proof}
Recall that the Chern character of $S$ is 
$
\ch(S)=(2, -H, 0, \frac{1}{12}H^{3}).
$
Since $\ch_{0}(S)=2\neq 0$,
by Theorem \ref{tilt-wall-struct-thm}, 
there exists a unique numerical vertical wall given by $\beta=-\frac{1}{2}$, 
and the numerical semicircular walls are two sets of nested semicircles whose apexes lie on the hyperbola
$$
(\beta+\frac{1}{2})^{2}-\alpha^{2}=(\frac{1}{2})^{2}.
$$
Therefore, all the numerical semicircular walls left to the vertical wall intersects with the vertical line $\beta=-1$.

It is sufficient to show that there are no walls for $S$ along $\beta=-1$.
Since $S$ is $\mu_{H}$-slope stable and $\mu_{H}(S)=-\frac{1}{2}>-1$,
by Corollary \ref{limit-tilt-stab-Gie-stab}, $S$ is $\sigma_{\alpha,-1}$-stable for $\alpha\gg 0$.
Since $H^{2}\cdot \ch_{1}^{-1}(S)=H^{3}$, 
any destabilizing subobject $F\subset S$ along $\beta=-1$ must satisfy $H^{2}\cdot \ch_{1}^{-1}(F)=H^{3}$ or $H^{2}\cdot \ch_{1}^{-1}(F)=0$.
Therefore, either subobject $F$ or the quotient object $E/F$ have infinite tilt slope, this is impossible.
It yields that $S$ is $\sigma_{\alpha,-1}$-stable for $\alpha>0$.
Hence, $S$ is $\sigma^{0}_{\alpha,\beta}$-stable for $(\alpha,\beta)\in V$ and $\beta<-\frac{1}{2}$.
Consequently, $S$ is $\sigma(\alpha,\beta)$-stable for $(\alpha,\beta)\in V$ and $\beta<-\frac{1}{2}$.
By Lemma \ref{GL-one-orbitV+}, $S$ is $\sigma$-stable for all $\sigma\in \mathcal{K}$.
\end{proof}

Naturally, one may ask the following:

\begin{quest}\label{Append-quest}
Let $E\in \Ku(Q)$ be a $\sigma$-stable object with class $\lambda_{2}=[S]\in \mathcal{N}(\Ku(Q))$ for $\sigma\in \mathcal{K}$.
Is, up to a shift, $E$ isomorphic to the spinor bundle $S$?
\end{quest}

The next goal is to give a confirmation to the above question.
To this end, we will gain the control over the maximal Chern characters $\ch_{2}$ and $\ch_{3}$ of $\mu_{H}$-slope stable sheaves with $\ch_{\leq 1}=(2,-H)$.

\begin{lem}\label{spinor-reflexive}
Take a $\mu_{H}$-slope stable sheaf $E$ on $Q$ with Chern character 
$$
\ch(E)=(2,-H,\ch_{2},\ch_{3}).
$$
Then $H\cdot \ch_{2}(E)\leq 0$; and if $H\cdot \ch_{2}(E)= 0$, then $\ch_{3}\leq \frac{1}{12}H^{3}$.
In particular, any $\mu_{H}$-slope stable sheaf with Chern character $(2, -H, 0, \frac{1}{12}H^{3})$ is a reflexive sheaf.
\end{lem}

\begin{proof}
Since the sheaf $E$ is $\mu_{H}$-slope stable,
by the classical Bogomolov inequality, we have $H\cdot \ch_{2}(E)\leq \frac{1}{4}H^{3}$.
By Lemma \ref{Num-Chow-Q}, it follows that $H\cdot \ch_{2}(E)\leq 0$

Suppose now that $H\cdot \ch_{2}(E)= 0$.
Since $\mu_{H}(E)=-\frac{1}{2}<0$, 
by $\mu_{H}$-slope stability, 
we get $\Hom(\CO,E)=0$; 
similarly, by Serre duality, we have $\Hom(\CO,E[3])\cong \Hom(E,\CO(-3H))=0$.
By the Hirzebruch--Riemann--Roch formula \eqref{HRR-quadric3}, 
we have 
\begin{equation}\label{Apend-eq1}
\chi(E)= \ch_{3}(E)-\frac{1}{12}H^{3}=-\hom(\CO,E[1])+\hom(\CO,E[2]).
\end{equation}
By Serre duality, we have $\Hom(\CO,E[2])\cong \Hom(E,\CO(-3H)[1])$.
Similarly as the proof of Lemma \ref{spinor-bd-stable}, 
we obtain that $E$ is $\sigma_{\alpha,-1}$-stable for $\alpha>0$.
By direct computations, we have
$$
\lim_{\alpha \to 0^{+}} \mu_{\alpha,-1}(E)=0 > -1=\lim_{\alpha \to 0^{+}} \mu_{\alpha,-1}(\CO(-3H)[1]),
$$
and thus $\Hom(E,\CO(-3H)[1])=0$. 
By the equation \eqref{Apend-eq1}, 
we get 
$$
\ch_{3}(E)-\frac{1}{12}H^{3}=-\hom(\CO,E[1])\leq 0
$$
and thus $\ch_{3}(E)\leq \frac{1}{12}H^{3}$.

Finally, assume that $E$ is a $\mu_{H}$-slope stable sheaf with the Chern character $(2, -H, 0, \frac{1}{12}H^{3})$.
Since $E$ is torsion-free,
there is a short exact sequence 
$$
\xymatrix@C=0.5cm{
0 \ar[r]^{} & E
\ar[r]^{} & E^{\vee\vee}
\ar[r]^{} & T \ar[r]^{} & 0,} 
$$
where $T$ is a sheaf supported in codimension $\geq 2$.
Note that $E^{\vee\vee}$ is also $\mu_{H}$-slope stable.
Applying the first statement, 
the second Chern character $\ch_{2}(E^{\vee\vee})=\ch_{2}(T) \leq 0$ and thus $\ch_{2}(E^{\vee\vee})=\ch_{2}(T)=0$.
As a result, we get the third Chern character 
$\ch_{3}(E^{\vee\vee})=\ch_{3}(T)+\frac{1}{12}H^{3}\leq \frac{1}{12}H^{3}$ and thus $\ch_{3}(T)=0$
Hence $T=0$ and thus $E\cong E^{\vee\vee}$.
So $E$ is a reflexive sheaf.
This completes the proof.
\end{proof}

\begin{prop}\label{semistab-char-spinorbd}
If $G$ is $\sigma^{0}_{\alpha,\beta}$-semistable with Chern character 
$$
\ch(G)=\pm(2,-H,0,\frac{1}{12}H^{3}),
$$
then $G$ is isomorphic to $S$ or $S[1]$. 
\end{prop}

\begin{proof}
Since $G$ is $\sigma^{0}_{\alpha,\beta}$-semistable object, 
by \cite[Proposition 4.1]{FP23}, 
we have the following cases:
\begin{enumerate}
  \item[(i)] $G$ is $\sigma_{\alpha,\beta}$-semistable; 
  \item[(ii)]  there is an exact  triangle $\xymatrix@C=0.3cm{F[1] \ar[r]^{} & G  \ar[r]^{} & T}$,
where $F\in \mathcal{F}^{0}_{\alpha,\beta}$ and $T\in \mathcal{T}^{0}_{\alpha,\beta}$ is supported on points.  
\end{enumerate}
For the case (ii), we may suppose that $\ch(G)=-(2,-H,0,\frac{1}{12}H^{3})$.
Since $G$ is $\sigma_{\alpha,\beta}^{0}$-semistable,
it follows that $F$ must be $\sigma_{\alpha,\beta}$-semistable; 
otherwise, it will contradict to the $\sigma_{\alpha,\beta}^{0}$-stability of $G$.
Hence, we may assume that the Chern character
$$
\ch(F)=(2,-H,0,(\frac{1}{12}+\frac{t}{2})H^{3}),
$$
where $t\geq 0$. 
Likewise to Lemma \ref{spinor-bd-stable}, 
there is no wall for objects with $\ch_{\leq 2}(F)=(2,-H,0)$ to the left of the vertical wall $\beta=-\frac{1}{2}$ and thus $F$ is $\sigma_{\alpha,\beta_{0}}$-semistable for $\alpha>0$.
Since $\mu_{H}(F)=-\frac{1}{2}>\beta_{0}$, 
by Corollary \ref{limit-tilt-stab-Gie-stab}, $F$ is a $\mu_{H}$-slope stable torsion-free sheaf.
By Lemma \ref{spinor-reflexive}, $\ch_{3}(F)=(\frac{1}{12}+\frac{t}{2})H^{3}\leq \frac{1}{12}H^{3}$ and thus $t=0$, as $t\geq 0$. 
This means that $T=0$ and thus $G=F[1]$. 
We are reduced to the case (i) for discussing $\sigma_{\alpha,\beta}$-semistable sheaf $F$.

For the case (i), we may assume $\ch(G)=(2,-H,0,\frac{1}{12}H^{3})$.
By hypothesis, $G$ is $\sigma_{\alpha,\beta}$-semistable.
Analogous to Lemma \ref{spinor-bd-stable}, 
there is no wall for objects with $\ch_{\leq 2}(G)=(2,-H,0)$ to the left of the vertical wall $\beta=-\frac{1}{2}$.
As a consequence, $G$ is $\sigma_{\alpha,\beta_{0}}$-semistable for $\alpha>0$.
By Corollary \ref{limit-tilt-stab-Gie-stab},
$G$ is a $\mu_{H}$-slope stable torsion-free sheaf.
Thanks to Lemma \ref{spinor-reflexive}, $G$ is a reflexive sheaf.
As a result, according to \cite[Lemma 41]{CJ24},  
$G$ is isomorphic the spinor bundle $S$.
\end{proof}

Now we are in the position to give an affirmative answer to Question \ref{Append-quest}.

\begin{thm}\label{spinor-bundle-Bstable}
If $E\in \Ku(Q)$ is a $\sigma$-stable object with class $\lambda_{2}$ for $\sigma\in \mathcal{K}$,
then up to a shift, $E$ is isomorphic to the spinor bundle $S$.
In particular, the moduli space $M_{\sigma}(\lambda_{2})=\{S\}$.
\end{thm}

\begin{proof}
The idea of the proof is the same as that of Proposition \ref{charact-stabob-inKu}.
By direct computations, the Chern character of $E$ is $\ch(E)=\pm(2,-H,0,\frac{1}{12}H^{3})$.
By Lemma \ref{GL-one-orbitV+},
we may assume that $E\in \Ku(Q)$ is $\sigma(\alpha,\beta)$-stable object with class $\lambda_{2}$ for $(\alpha,\beta)\in V$ and $\beta<-\frac{1}{2}$.
Therefore, up to a shift $k\in \ZN$, $E[k]\in \A(\alpha,\beta)$ is $\sigma(\alpha,\beta)$-stable with the slope
$$
\mu_{\alpha,\beta}^{0}(E[k])=\frac{-(2\beta+1)}{-\beta-\beta^{2}+\alpha^{2}}.
$$
We set $G:=E[1]$.
Note that the limit of $\mu_{\alpha,\alpha-1}^{0}$-slopes
$$
\displaystyle
\lim_{\alpha \to 0^{+}} \mu_{\alpha,\alpha-1}^{0}(G)
= \lim_{\alpha \to 0^{+} } (\frac{1}{\alpha}-2)=+\infty
$$
for $0<\alpha<\frac{1}{2}$ and $(\alpha,\alpha-1)\in V$.
By Proposition \ref{semistab-char-spinorbd}, 
it is sufficient to show that  $G$ is $\sigma_{\alpha,\beta}^{0}$-semistable for some $(\alpha_{0},\beta_{0})\in V_{L}$.

Suppose that $G$ is not $\sigma_{\alpha,\beta}^{0}$-semistable for all $(\alpha,\beta)\in V_{L}$.
Since $G$ is $\sigma(\alpha,\beta)$-stable, 
by \cite[Proposition 2.2.2]{BMT14}, 
there is an open ball $U$ in $\RN^{2}$ containing the point $(0,-1)$ such that for any $(\alpha,\beta)\in U\cap V_{L}$,  $G\in \A(\alpha,\beta)$ and $G$ has constant HN filtration with respect to $\sigma^{0}_{\alpha,\beta}$.
Suppose that $B$ is the destabilizing quotient object of $G$ with minimum slope and 
$A\to G\to B$ 
is the destabilizing sequence of $E$ with respect to $\sigma_{\alpha,\beta}^{0}$ for $(\alpha,\beta)\in U\cap V_{L}$. 
Likewise to the proof of Claim \ref{claim-for-charact-stabob-inKu},
we may assume that $\ch_{\leq 2}(B)=(a,bH,\frac{-a-2b}{2}H^2)$ with $a,b\in \ZN$.
Then we have
\begin{equation}\label{claim-im-ineq-S}
0<(\beta^{2}-\alpha^{2}-1)\frac{a}{2}-(\beta+1)b\leq -\beta-\beta^{2}+\alpha^{2}
\end{equation} 
\begin{equation}\label{claim-mu-slope-ineq-S}
\frac{(\beta a-b)}{(\beta^{2}-\alpha^{2}-1)\frac{a}{2}-(\beta+1)b}
<
\frac{-(2\beta+1)}{-\beta-\beta^{2}+\alpha^{2}}
\end{equation}
for all $(\alpha,\beta)\in U\cap V$.
Combining \eqref{claim-im-ineq-line} with  \eqref{claim-mu-slope-ineq-line}, 
we get
\begin{equation}\label{claim-ineq-comb-S}
 \beta a-b < - (2\beta+1).
\end{equation}
According to \cite[Remark 5.12]{BLMS23}, 
one has $\mu_{\alpha,\beta}^0(B) \geq -2$ for $(\alpha,\beta)\in V_{L}$.
Considering the limit cases of \eqref{claim-mu-slope-ineq-S} and \eqref{claim-ineq-comb-S} 
when $(\alpha,\beta) \to (0,-1)$, we have $a+b=0$ or $a+b=-1$.

If $a+b=0$, this contradicts to \eqref{claim-mu-slope-ineq-S}.

If $a+b=-1$, then $\ch_{\leq 2}(B)=(a,(-a-1)H,\frac{a+2}{2}H^{2})$.
By  \eqref{claim-ineq-comb-S} and $\beta>-1$, 
we have  $a<-2$.
Since $B$ is $\sigma_{\alpha,\beta}^{0}$-semistable, 
by \cite[Proposition 4.1]{FP23}, 
$\ch_{\leq 2}(B)$ is a possible truncated Chern character of a $\sigma_{\alpha,\beta}$-semistable object $B'[1]$ for $B' \in \Coh^{\beta}(Q)$.
By \cite[Proposition 3.2]{Li19}, we have $|a|\leq 2$, a contradiction.
This finishes the proof of Theorem \ref{spinor-bundle-Bstable}.
\end{proof}

\section{Hilbert scheme of lines}\label{Append-Hilb-line}

In this appendix, we will reinterpret the Hilbert scheme of lines in $Q$ as a Bridgeland moduli space in $\Ku(Q)$.

By the Hartshorne--Serre correspondence,
one can view the spinor bundle $S$ as a bundle associated with a line $\ell\subset Q$;
namely, there is a short exact sequence 
$$
\xymatrix@C=0.5cm{
0 \ar[r]^{} & \CO(-H) \ar[r]^{} & S \ar[r]^{} & I_{\ell} \ar[r]^{} & 0,}
$$
where $I_{\ell}$ is the ideal sheaf of $\ell$ in $Q$.
Therefore, the ideal sheaf $I_{\ell}\in \Ku(Q)$, the Chern character $\ch(I_{\ell})=(1,0,-\frac{1}{2}H^{2},\frac{1}{4}H^{3})$ and the numerical class $[I_{\ell}]=\lambda_{2}-\lambda_{1}\in \mathcal{N}(\Ku(Q))$.

\begin{lem}\label{ideal-line-stab}
The ideal sheaf $I_{\ell}$ is $\sigma$-stable for all stability conditions $\sigma\in \mathcal{K}$.
\end{lem}

\begin{proof}
Note that $\ch_{0}(I_{\ell})\neq 0$, Theorem \ref{tilt-wall-struct-thm} yields that there is a unique numerical vertical wall given by $\beta=0$, and the numerical semicircular walls are two sets of nested semicircles whose apexes lie on the hyperbola $\beta^{2}-\alpha^{2}=1$. 
Then all the numerical semicircular walls to the left of the vertical wall intersects with $\beta=-1$.
Hence, it suffices to prove that there are no walls for $I_{\ell}$ along the vertical line $\beta=-1$.
Since $I_{\ell}$ is $\mu_{H}$-slope stable with $\mu_{H}(I_{\ell})=0>-1$,
so $I_{\ell}\in \Coh^{-1}(Q)$ and it is $\sigma_{\alpha,-1}$-stable for $\alpha\gg0$ from Proposition \ref{limit-tilt-stab-Gie-stab-prop}.
Since $H^{2}\cdot \ch_{1}^{-1}(I_{\ell})=H^{3}$,
any destabilizing subobject $F\subset I_{\ell}$ along $\beta=-1$ must have $H^{2}\cdot \ch_{1}^{-1}(F)=H^{3}$ or $H^{2}\cdot \ch_{1}^{-1}(F)=0$.
As a result, either subobject $F$ or the quotient object $I_{\ell}/F$ have infinite tilt slope, a contradiction.
This implies that $I_{\ell}$ is $\sigma_{\alpha,-1}$-stable for $\alpha>0$.
Therefore, $I_{\ell}$ is $\sigma(\alpha,\beta)$-stable for $(\alpha,\beta)\in V$.
In particular, by Lemma \ref{GL-one-orbitV+},  $I_{\ell}$ is $\sigma$-stable for all $\sigma\in \mathcal{K}$.
\end{proof}

The above lemma implies that the moduli space $M_{\sigma}(\lambda_{2}-\lambda_{1})$ is non-empty for every $\sigma\in \mathcal{K}$.
Let $\mathrm{Hilb}_{Q}^{m+1}$ be the Hilbert scheme of lines in $Q$. 
It is well-known that the Hilbert scheme $\mathrm{Hilb}_{Q}^{m+1}$ is isomorphic to $\PB^{3}$.
Next we give a modular description to this Hilbert scheme.

\begin{thm}\label{Hilb-line-Q}
The moduli space $M_{\sigma}(\lambda_{2}-\lambda_{1})$ is isomorphic to the Hilbert scheme $\mathrm{Hilb}_{Q}^{m+1}$.
\end{thm}

\begin{proof}
Using the same idea as the proof of Proposition \ref{main-thm-iso-prop3}; see also \cite[Theorem 4.5]{FP23} and \cite[Theorem 1.1]{PY22},
it suffices to show that for any $\sigma(\alpha,\beta)$-stable object $E\in \Ku(Q)$ with class $\lambda_{2}-\lambda_{1}$, up to a shift, it is isomorphic to the ideal sheaf $I_{\ell}$ for some line $\ell\subset Q$.
We may assume that $G:=E[1]\in \A(\alpha,\beta)$ is $\sigma(\alpha,\beta)$-stable for $(\alpha,\beta)\in V$.
Then the slope 
$
\displaystyle
\mu_{\alpha,\beta}^{0}(G)=\frac{-\beta}{\frac{1}{2}-\frac{1}{2}\beta^{2}+\frac{1}{2}\alpha^{2}}.
$
Then  $
\displaystyle
\lim_{\alpha \to 0^{+}} \mu_{\alpha,\alpha-1}^{0}(G)
= \lim_{\alpha \to 0^{+} } (\frac{1}{\alpha}-1)=+\infty.$

Next we show that $G$ is $\sigma^{0}_{\alpha,\beta}$-semistable for some $(\alpha,\beta)\in V$.
Suppose that $G$ is not $\sigma_{\alpha,\beta}^{0}$-semistable for all $(\alpha,\beta)\in V$.
Since $G$ is $\sigma(\alpha,\beta)$-stable, 
by \cite[Proposition 2.2.2]{BMT14}, 
there is an open ball $U$ in $\RN^{2}$ containing the point $(0,-1)$ such that for any $(\alpha,\beta)\in U\cap V$, 
$G\in \A(\alpha,\beta)$ and $G$ has constant HN filtration with respect to $\sigma^{0}_{\alpha,\beta}$.
Suppose that $B$ is the destabilizing quotient object of $G$ with minimum slope and 
$A\to G\to B$ 
is the destabilizing sequence of $E$ with respect to $\sigma_{\alpha,\beta}^{0}$ for $(\alpha,\beta)\in U\cap V$. 
Likewise to the proof of Claim \ref{claim-for-charact-stabob-inKu},
we may assume that $\ch_{\leq 2}(B)=(a,bH,\frac{-a-2b}{2}H^2)$ with $a,b\in \ZN$.
Then we have
\begin{equation}\label{claim-im-ineq-line}
0<(\beta^{2}-\alpha^{2}-1)\frac{a}{2}-(\beta+1)b\leq \frac{1}{2}-\frac{1}{2}\beta^{2}+\frac{1}{2}\alpha^{2}
\end{equation} 
\begin{equation}\label{claim-mu-slope-ineq-line}
\frac{(\beta a-b)}{(\beta^{2}-\alpha^{2}-1)\frac{a}{2}-(\beta+1)b}
<
\frac{-\beta}{\frac{1}{2}-\frac{1}{2}\beta^{2}+\frac{1}{2}\alpha^{2}}
\end{equation}
for all $(\alpha,\beta)\in U\cap V$.
Combining \eqref{claim-im-ineq-line} with  \eqref{claim-mu-slope-ineq-line}, 
we get
\begin{equation}\label{claim-ineq-comb-line}
 \beta a-b < - \beta.
\end{equation}
According to \cite[Remark 5.12]{BLMS23}, 
one has $\mu_{\alpha,\beta}^0(B) \geq -2$ for $(\alpha,\beta)\in V_{L}$.
Considering the limit cases of \eqref{claim-mu-slope-ineq-line} and \eqref{claim-ineq-comb-line} when $(\alpha,\beta) \to (0,-1)$, we have $a+b=0$ or $a+b=-1$.

If $a+b=0$, then it contradicts to \eqref{claim-mu-slope-ineq-line}.  

If $a+b=-1$, then $\ch_{\leq 2}(B)=(a,(-a-1)H,\frac{a+2}{2}H^{2})$.
Since $\beta>-1$, using \eqref{claim-ineq-comb-line},
it follows that $a<-1$.
Since $B$ is $\sigma_{\alpha,\beta}^{0}$-semistable, 
by \cite[Proposition 4.1]{FP23}, 
$\ch_{\leq 2}(B)$ is a possible truncated Chern character of a $\sigma_{\alpha,\beta}$-semistable object $B'[1]$ for $B' \in \Coh^{\beta}(Q)$.
By \cite[Proposition 3.2]{Li19}, we have $|a|\leq 2$ and thus $a=-2$.
Therefore, we get $\ch_{\leq 2}(B)=(-2,H,0)$ and $\ch_{\leq 2}(A)=(1,-H,\frac{1}{2}H^{2})$.
Analogous to Claim \ref{ch-B-part}, we have $\ch(B)=(-2,H,0,-\frac{1}{12}H^{3})$ and thus $\ch(A)=(1,-H,\frac{1}{2}H^{2},-\frac{1}{6}H^{3})$.
By Proposition \ref{semistab-char-spinorbd}, we get $B\cong S[1]$.
Note that $\ch(A(H))=(1,0,0,0)$.
Since $A$ is $\sigma^{0}_{\alpha,\beta}$-semistable for  $(\alpha,\beta)\in U\cap V$, 
by \cite[Proposition 4.1]{FP23}, $A$ is $\sigma_{\alpha,\beta}$-semistable.
By Proposition \ref{sum-lines-detect}, we get $A(H)\cong \CO$ and thus $A\cong \CO(-H)$.
Since $\RHom(S,\CO(-H))=0$, it follows that $G\cong \CO(-H)\oplus S[1]$, a contradiction.

Now we may assume that $G$ is $\sigma^{0}_{\alpha,\beta}$-semistable for some $(\alpha,\beta)\in V$.
By \cite[Proposition 4.1]{FP23}, 
there are two cases: (i) $G$ is $\sigma_{\alpha,\beta}$-semistable; 
(ii) there is an exact  triangle $\xymatrix@C=0.3cm{F[1] \ar[r]^{} & G \ar[r]^{} & T}$,
where $F\in \mathcal{F}^{0}_{\alpha,\beta}$ and $T\in \mathcal{T}^{0}_{\alpha,\beta}$ is supported on points.  
For the case (ii), we may suppose that $G=E[1]$ and thus $\ch_{\leq 2}(F)=(1,0,-\frac{1}{2}H^{2})$.
Since $G$ is $\sigma_{\alpha,\beta}^{0}$-semistable and thus $F$ must be $\sigma_{\alpha,\beta}$-semistable.
We may assume that the Chern character
$$
\ch(F)=(1,0,-\frac{1}{2}H^{2},(\frac{1}{4}+\frac{t}{2})H^{3}),
$$
where $t\geq 0$. 
Likewise to Lemma \ref{ideal-line-stab}, 
there is no wall for objects with $\ch_{\leq 2}(F)=(1,0,-\frac{1}{2}H^{2})$ to the left of the vertical wall $\beta=0$ and thus $F$ is $\sigma_{\alpha,\beta}$-semistable for $\alpha>0$.
Since $\mu_{H}(F)=0>\beta$, 
by Corollary \ref{limit-tilt-stab-Gie-stab}, $F$ is a $\mu_{H}$-slope stable torsion-free sheaf.
As proved in Lemma \ref{ideal-line-stab}, 
$F$ is $\sigma_{\alpha,-1}$-stable for any $\alpha > 0$. 
Note that 
$
\displaystyle \lim_{\alpha\to 0^{+}}\mu_{\alpha, -1}(\CO(-3H)[1])=-1< 0=\lim_{\alpha\to 0^{+}} \mu_{\alpha, -1}(F).
$
By Serre duality, we have $\Ext^{2}(\CO,F)\cong \Hom(F,\CO(-3H))=0$.
Since $G\in \Ku(Q)$, it follows that $\Hom(\CO,F)=0$.
By \eqref{HRR-quadric3}, we get $\chi(F)=\ch_{3}(F)-\frac{1}{4}H^{3}\leq 0$.
As a result, 
we have $\ch_{3}(F)=(\frac{1}{4}+\frac{t}{2})H^{3} \leq \frac{1}{4}H^{3}$ and thus $t=0$, as $t\geq 0$. 
This yields that $T=0$ and thus $G=F[1]$.
We are reduced to the case (i).

For the case (i), we may assume $G:=E$.
By hypothesis, $G$ is $\sigma_{\alpha,\beta}$-semistable.
Similar to Lemma \ref{ideal-line-stab}, 
there is no wall for objects with $\ch_{\leq 2}(G)=(1,0,-\frac{1}{2}H^{2})$ to the left of the vertical wall $\beta=0$.
As a consequence, $G$ is $\sigma_{\alpha,\beta}$-semistable for $\alpha>0$.
By Corollary \ref{limit-tilt-stab-Gie-stab},
$G$ is a $\mu_{H}$-slope stable torsion-free sheaf.
This completes the proof of Theorem \ref{Hilb-line-Q}.
\end{proof}


\end{document}